\documentclass[11pt]{amsart}
\usepackage{amsmath}
\usepackage{amscd}
\usepackage{amsthm, amssymb, amsfonts, latexsym}
\usepackage[all]{xy}
\usepackage{graphics}
\usepackage{graphicx}
\usepackage{hyperref}

\def\ker{\mathop{\rm Ker}}%
\def\SL{\mathop{\rm SL}}
\def\Re{\mathop{\rm Re}}

\def\det{\mathop{\rm det}}%
\def\tr{\mathop{\rm Tr}}%
\def\dim{\mathop{\rm dim}}%

\def\supp{\mathop{\rm supp}}%
\def\dom{\mathop{\rm Dom}}%
\def\diver{\mathop{\rm div}}%
\def\grad{\mathop{\rm grad}}%

\newcommand{\wt}{\widetilde}%

\def\R{\mathop{\mathbb R}}
\def\C{\mathop{\mathbb C}}
\def\N{\mathop{\mathbb N}}
\def\Z{\mathop{\mathbb Z}}
\def\H{\mathop{\mathbb H}}
\def\M{\mathop{\mathcal M}}

\newtheorem*{acknowledgements}{Acknowledgements}
\newtheorem{theorem}{Theorem}

\newtheorem{definition}[theorem]{Definition}

\newtheorem{proposition}[theorem]{Proposition}
\newtheorem{remark}[theorem]{Remark}

\newtheorem{lemma}[theorem]{Lemma}

\begin{document}

\title[Isoresonant surfaces with cusps and relative determinants]{Isoresonant conformal surfaces with cusps and boundedness of the relative determinant}
\author{Clara L. Aldana}

\address{Mathematical Institute, University of Bonn}
\email{clara.aldana@gmail.com}

\maketitle

\begin{abstract}
We study the isoresonance problem on non-compact surfaces of finite area that are hyperbolic outside a compact set.
Inverse resonance problems correspond to inverse spectral problems in the non-compact setting.
We consider a conformal class of surfaces with hyperbolic cusps
where the deformation takes place inside a fixed compact set. Inside this compactly supported conformal class
we consider isoresonant me\-trics, i.e. metrics for which the set of resonances is the same, including multiplicities.
We prove that sets of isoresonant metrics inside the conformal class are sequentially compact.
We use relative determinants, splitting formulae for determinants and the result of
B. Osgood, R. Phillips and P. Sarnak about compactness of sets of isospectral metrics on
closed surfaces.

In the second part, we study the relative determinant of the Laplace operator on a hyperbolic surface as function on the
moduli space. We consider the moduli space of hyperbolic surfaces of fixed genus and fixed number of cusps.
We consider the relative determinant of the Laplace operator and a model operator defined on the cusps. We prove that
the relative determinant tends to zero as one approaches the boundary of the moduli space.
\end{abstract}

\section{Introduction}

In this paper we consider two problems. We first focus on the isoresonance problem for a surface with
cusps and negative Euler characteristic, restricting our attention to a suitable conformal class of metrics.
The second problem, in Section \ref{section:boundedness},
is the study of the relative determinant of the Laplacian compared to a fixed model operator on
the moduli space of hyperbolic surfaces of fixed conformal type as one approaches the boundary.\\

We study the inverse resonance problem inside a conformal class of me\-trics whose
\lq\lq conformal factors\rq\rq \ have support in a
fixed compact set. We prove that given a fixed compact set $K\subset M$, inside a \lq\lq
$K$-compactly supported\rq\rq \ conformal class, sets of isoresonant
metrics are compact in the $C^{\infty}$-topology.
With this we partially generalize the result of B. Osgood,
R. Phillips, and P. Sarnak (OPS) in \cite{OPS2} that states that on a
closed surface every set of isometry classes of isospectral metrics
is sequentially compact in the $C^{\infty}$-topology.
We use the results of W. M\"uller about scattering theory for admissible surfaces in \cite{Mu2}.

Isospectral problems go back to 1960 when Leon Green asked if a
Rie\-mannian manifold was determined by its spectrum. The question
was re\-phra\-sed by Kac for planar domains in the very suggestive way:
\lq\lq Can one hear the shape of a drum?\rq \rq\  see \cite{Kac}. An
important result is the well known existence of non-isometric
manifolds that are isospectral, see \cite{Sunada} and the references
therein. We also refer to \cite{Zelditch} for a comprehensive survey
of inverse spectral problems in geometry.

The compactness theorem of OPS in
\cite{OPS2} uses the fact that if two me\-trics are isospectral,
i.e. spectra of the Laplacians are the same including multiplicities, then the heat
invariants and the determinant of the Laplace
operator have the same values at each metric. The authors note in the paper that the use of the regularized
determinant of the Laplacian is essential in order to obtain
compactness, since the heat invariants are not enough. On planar
domains the problem has been studied by R. Melrose in \cite{Melrose}
and OPS in \cite{OPS3}, and for compact surfaces with boundary by Y.
Kim in \cite{Kim}.

The isospectral problem also makes sense for certain non-compact
ma\-ni\-folds. Then scattering theory comes into play and we need to
deal with inverse scattering theory. The spectrum of the Laplacian is not enough,
one also has to consider resonances.
For example, on exterior planar
domains the isospectral problem was studied by A. Hassell and S.
Zelditch in \cite{HassellZelditch}. There two exterior planar
domains are called isophasal if they have the same scattering phase.
Hassell and Zelditch prove that each class of isophasal exterior
planar domains is sequentially compact in the $C^{\infty}$-topology.
In the proof they define a regularized
determinant of the Laplacian that plays a fundamental role.
More recently, D. Borthwick, C. Judge, and P.A. Perry in \cite{BorJudPer}
used determinants to prove sequential compactness of sets of isopolar (same scattering phase) surfaces of infinite volume under certain conditions.
In a later work \cite{BorthPerry}, Borthwick and Perry studied the inverse resonance pro\-blem for infinite volume manifolds
of finite dimension that are hyperbolic outside a compact set. In dimension $2$ they improved
the result of \cite{BorJudPer} and proved compactness of isoresonant surfaces without cusps that are isometric at infinity.

We study the isospectral problem inside a conformal class of a given
me\-tric in a surface with cusps. In this setting, two metrics are
isospectral if the resonances are the same for both metrics
including multiplicities. Because of this we use the terminology
\lq\lq isoresonant\rq\rq \ instead of \lq\lq isospectral\rq\rq.
For hyperbolic surfaces of finite area, W.
M\"uller proved in \cite{Mu2} that the resonance set associated to
the surface determines the surface up to finitely many
possibilities. Our result in this part is the following theorem:\\

{\bf Theorem \ref{theorem:compactnessisospectral}} \textit{ Let
$(M,g)$ be a surface with cusps with negative Euler characteristic, $\chi(M)<0$, let $K\subset M$ be compact, and let
$[g]_{K}= \{e^{2\varphi}g\ \vert \ \varphi \in C_{c}^{\infty}(M),\
\supp(\varphi)\subset K\}$ be the $K$-compactly supported conformal
class of $g$. Then
isoresonant sets in $[g]_{K}$ are compact in the $C^{\infty}$-topology.}\\

There are two strong restrictions in this theorem. First, we consider deformations only
with compact support because of the lack of results in the theory of resonances of surfaces with asymptotically hyperbolic
cusp ends.
The second restriction is to consider deformations only inside a conformal class. This is due to the fact that
the proof of Theorem \ref{theorem:compactnessisospectral}
relies on a splitting formula for the relative determinant to reduce the
problem to the compact case. The splitting formula relates
$\det(\Delta_{g},\Delta_{\beta,0})$ (with $\beta$ big enough) to the
determinant of the Dirichlet-to-Neumann operator acting on a
submanifold of $M$ homeomorphic to $S^{1}$. To relate the determinants of the Dirichlet-to-Neumann maps associated to
different metrics we use the conformal variation of the Laplacian. If the metrics are not conformal, it is not
clear how the different Dirichlet-to-Neumann operators are related. We prove this formula in
Section \ref{section:splitting}.
The main difficulty to treat the isoresonance problem on surfaces with cusps is that the injectivity radius
of these surfaces vanishes and the Sobolev embeddings do not hold anymore.\\

In the second part we study the relative determinant as a function on the
moduli space of hyperbolic surfaces with cusps. We work over $\M_{p,m}$, the
moduli space of compact Riemann surfaces of genus $p$ with $m$ punctures.
Each such a surface can be decomposed as the union of a compact part and $m$ cusps, as it is explained at
the beginning of sections \ref{subsection:swc} and \ref{section:boundedness}.
In order to define the relative determinant, we use a global model operator. We define the free Laplacian as being the Dirichlet Laplacian $\bar{\Delta}_{1,0}$, as in
Definition \ref{def:Laplmcuspas}, associated to the union of $m$ cusps all starting at $1$, i.e.
each cusp is taken as $[1,\infty)\times S^{1}$ with the hyperbolic metric on it.
In particular, $\bar{\Delta}_{1,0}$ is independent of $[g]$.
Hence the relative determinant
defines a function on the moduli space: $[g]\in
{\M}_{p,m}\mapsto\det(\Delta_{g},\Bar{\Delta}_{1,0})\in{\R}^+$, where $g\in[g]$
is hyperbolic. We prove Theorem \ref{theorem:vdbMs} that establishes that the relative determinant
$\det(\Delta_{g},\bar{\Delta}_{1,0})$ tends to zero as $[g]$ approaches the boundary of the moduli space.
Points at the boundary of $\M_{p,m}$ can be reached through a degene\-ra\-ting
family of metrics. The degeneration arises from closed geodesics whose length
converges to zero. The proof of Theorem \ref{theorem:vdbMs} relies strongly on the results of
L. Bers in \cite{Bers} and of J. Jorgenson and R. Lundelius (JL) in \cite{JorLund1}.
We remark that the hyperbolic determinant of JL also tends to zero as the metric approaches the boundary
of the moduli space. However, they do not state it explicitly in \cite{JorLund1}.
In an earlier work \cite{Lundelius}, Lundelius considers a relative determinant for admissible surfaces.
He stu\-dies the behavior of the relative weight (minus the logarithm of his relative determinant)
of a continuous family of hyperbolic surfaces of finite volume that degenerates by pinching geodesics; but
again there is no mention to the moduli space.
Our contribution in this part consists in using the results of \cite{JorLund1} and \cite{Bers}
to make a statement about the behavior of the relative determinant $\det(\Delta_{\cdot},\bar{\Delta}_{1,0})$
as function on the moduli space $\M_{p,m}$.
Although our remarks on this are straightforward consequences of
these results, they are worth mentioning explicitly in light of future
investigations on isospectral compactness problems.

\begin{acknowledgements}
This paper is based on my doctoral thesis that I completed at the University of Bonn in 2008.
I thank my supervisor W. M\"uller
for his continuous guidance through the project. I am grateful to R. Mazzeo for
helpful discussions, to the MSRI where I first started to write this paper, to the MPIM where I finally finished it, and to
the MI of the University of Bonn for hosting me during my studies.
Finally, I thank a referee whose comments helped me to improve this paper.
\end{acknowledgements}

\section{Surfaces with cusps, Laplacians, and relative determinants}
\label{subsection:swc}

A surface with cusps is a $2$-dimensional Riemannian manifold $(M,g)$ that is complete,
non-compact, has finite volume and is hyperbolic in the complement
of a compact set. It admits a decomposition of the form
\begin{equation*}
M=M_0\cup Z_{1}\cup\cdots\cup Z_{m},
\end{equation*}
where $M_0$ is a compact surface with smooth boundary and for each
$i=1,...,m$ we assume that
$$Z_i\cong [a_i,\infty)\times S^{1} \ni (y_i,x_i), \quad g|_{Z_i}=y_{i}^{-2}(dy_{i}^2+dx_{i}^2),\quad a_{i}>0.$$
The subsets $Z_{i}$ are called cusps. Sometimes we denote $Z_{i}$ by
$Z_{a_{i}}$\index{$Z_{a_{i}}$} to indicate the \lq \lq starting
point\rq \rq \ $a_{i}$. Instances of surfaces with cusps are
quotients of the form $\Gamma(N) \backslash {\mathbb H}$, where
${\mathbb H}$ is the upper half plane and $\Gamma(N)\subseteq
\SL_{2}(\Z)$ is a congruence subgroup.

To any surface with cusps $(M,g)$ we can associate a compact surface
$\overline{M}$ such that $(M,g)$ is diffeomorphic to the complement
of $m$ points in $\overline{M}$.
Let $p$ denote the genus of the compact surface $\overline{M}$; then
the pair $(p,m)$ is called the conformal type of $M$.

For any oriented Riemannian manifold $(M,g)$ the Laplace-Beltrami
ope\-ra\-tor on functions is defined as
$\Delta f = -\diver \grad f$. It is equal to $\Delta=d^{*}d$.
If we want to emphasize the dependence on the metric we
denote the Laplacian by $\Delta_{g}$. We consider positive Laplacians. If $(M,g)$ is
complete, $\Delta_{g}$ has a unique closed extension that is denoted in the same way.

Let us consider some Laplacians that are naturally associated to the cusps:
\begin{definition}\label{def:Laplmcuspas}
Let $a>0$, let $\Delta_{a,0}$\index{$\Delta_{a,0}$} denote the
self-adjoint extension of the operator
$$-y^{2} {\frac{\partial^{2}}{\partial y^{2}}}: C_{c}^{\infty}((a,\infty))\to L^{2}([a,\infty),y^{-2}dy)$$
obtained after imposing Dirichlet boundary conditions at $y=a$. The
domain of $\Delta_{a,0}$ is given by $\dom(\Delta_{a,0}) =
H_{0}^{1}([a,\infty)) \cap H^{2}([a,\infty))$, where
$H_{0}^{1}([a,\infty))=\{f\in H^{1}([a,\infty)): f(a)=0\}$.

Let $\bar{\Delta}_{a,0}=\oplus_{j=1}^m
\Delta_{a_{j},0}$\index{$\bar{\Delta}_{a,0}$} be defined as the
direct sum of the self-adjoint operators operators
$\Delta_{a_{j},0}$ defined above. The operator $\bar{\Delta}_{a,0}$
acts on a subspace of $\oplus_{j=1}^m
L^2([a_j,\infty),y_j^{-2}dy_j)$.
\end{definition}

The kernel of the heat operator associated to $\Delta_{a_{j},0}$ is described in \cite[sec.14.2]{CarslawJaeger}
and it is given by the equation:
\begin{equation}
p_{a}(y,y',t) = {\frac{e^{-t/4}}{\sqrt{4\pi t}}} \ (yy')^{1/2}
\left\{ e^{-(\log(y/y'))^{2}/4t} -
e^{-(\log(yy')-\log(a^{2}))^{2}/4t}\right\},\label{eq:psuba}\end{equation}
for $y, y'>a$, and for $1\leq y\leq a$, $p_{a}(y,y',t)=0$. We extend it in the obvious way to see it as
a function of $z\in M$.

Now, let $a>0$, let $Z_{a}$ be endowed with the hyperbolic metric
$g$ and let $\Delta_{Z_{a},D}$ be the
self-adjoint extension of
$$-y^{2}\left( {\frac{\partial^{2}}{\partial
y^{2}}}+{\frac{\partial^{2}}{\partial x^{2}}}\right):
C_{c}^{\infty}((a,\infty)\times S^{1})\to L^{2}(Z_{a},dA_{g})$$
obtained after imposing Dirichlet boundary conditions at
$\{a\}\times S^{1}$.

Let us describe a decomposition of the operator $\Delta_{Z_{a},D}$ that is very useful in our case:
The space $L^{2}(Z_{a},dA_{g})$ can be decomposed using the isomorphism
$$L^{2}(Z_{a},dA_{g}) \cong L^{2}([a,\infty),y^{-2}dy) \oplus L^{2}_{0}(Z_{a}),$$
with $L^{2}_{0}(Z_{a})=\{f\in
L^{2}(Z_{a},dA_{g})\vert \int_{S^{1}} f(y,x) dx = 0 \text{ for a. e.
} y\geq a\}$. This decomposition is invariant under $\Delta_{Z_{a},D}$; in terms
of it, we can write $\Delta_{Z_{a},D} = \Delta_{a,0}
\oplus \Delta_{Z_{a},1}$ where
$\Delta_{Z_{a},1}$\index{$\Delta_{Z_{a},1}$} acts on
$L^{2}_{0}(Z_{a})$.

For the spectral theory of manifolds with cusps
we refer the reader to W. M\"uller in \cite{Mu} and \cite{Mu2}, to Y. Colin de Verdi\`{e}re in \cite{ColinV}, and to
the references therein.
The results in \cite{Mu} hold for any dimension. For surfaces in
particular we refer to \cite{Mu2}.

On a surface with cusps $(M,g)$, the spectrum of the Laplacian
$\sigma(\Delta_{g})$ is the union of the point spectrum $\sigma_{p}$
and the continuous spectrum $\sigma_{c}$. The point spectrum consist
of a sequence of eigenvalues
$$0=\lambda_{0}<\lambda_{1}\leq \lambda_{2} \leq \dots $$
Each eigenvalue has finite multiplicity, and the counting function
$N(\Lambda)= \# \{\lambda_{j}\vert \lambda_{j}\leq \Lambda^{2}\}$
for $\Lambda>0$ satisfies $\limsup N(\Lambda) \Lambda^{-2} \leq
A_{g}(4\pi)^{-1}$, where $A_{g}$ denotes the area of $(M,g)$.
Depending on the metric, the set of eigenvalues may be infinite or
not.

The continuous spectrum $\sigma_{c}$ of $\Delta_{g}$ is the interval
$[{\frac{1}{4}},\infty)$ with multipli\-ci\-ty equal to the number of
cusps of $M$. The spectral decomposition of the absolutely
continuous part of $\Delta_{g}$ is described by the generalized
eigenfunctions $E_{j}(z,s)$, for $j=1,\dots, m$ with $z\in M$, $s\in
\C$. Let us recall some of their properties as well as the definition of the scattering matrix that
we will use; for the details see \cite{Mu} and \cite{Mu2}.
To each cusp there is associated a generalized eigenfunction that
satisfies:
$$\Delta_{g} E_{i}(z,s) = s(1-s) E_{i}(z,s), \quad \text{ for } z\in M.$$
Each $E_{i}(z,s)$ is a meromorphic function of $s\in \C$ with poles contained in
$\{s \ \vert \Re(s)<1/2\} \cup (1/2, 1]$.
The zeroth Fourier coefficient of the expansion of $E_{i}(s,z)$ in a Fourier series on the cusp
$Z_{j}=[a_{j}, \infty) \times S^{1}$
has the form
$$\delta_{ij} y_{j}^{s} + C_{ij}(s)y_{j}^{1-s}, \quad \text{ for } y_{j}\geq a_{j}.$$
Using this expression we can define the scattering matrix as the $m \times m$ matrix given by:
$$C(s)=(C_{ij}(s)).$$
It is a meromorphic function of $s\in \C$ and all its poles are contained in $\{s \ \vert \Re(s)<1/2\} \cup (1/2, 1]$.
The scattering matrix also satisfies:
\begin{equation*}
C(s)C(1-s)=\text{Id}, \quad \overline{C(s)}=C(\bar{s}), \quad \text{ and} \quad
C(s)^{*}=C(\bar{s}).
\end{equation*}

A quantity of interest is the determinant of the scattering matrix
which we denote by $\phi(s)=\det C(s).$ It satisfies the following
equations:
\begin{equation*}
\phi(s)\phi(1-s)=1, \quad \overline{\phi(s)}=\phi(\bar{s}), \quad
s\in \C.
\end{equation*}
The poles of $\phi(s)$ will be called resonances. They will be the
complementary quantities to the eigenvalues that we will need to study \lq\lq isospectral\rq \rq \ surfaces.

In \cite{Mu3} W. M\"uller defines the relative determinant for pairs of operators in a general setting. Let us recall the
definition since we will use it. Let $H_1$ and $H_0$ be two self-adjoint, nonnegative linear operators in a
separable Hilbert space ${\mathcal H}$ satisfying the following
assumptions:

\begin{enumerate}
\item For each $t>0$, $e^{-t H_1} - e^{-t H_0}$ is a trace class operator.
\item As $t\to 0$, there is an asymptotic expansion of the relative trace of the
form:
$$\tr(e^{-t H_1}-e^{-t H_{0}}) \sim \sum_{j=0}^{\infty}\sum_{k=0}^{k(j)} a_{jk} t^{\alpha_{j}}
\log^{k}t,$$ where $-\infty < \alpha_0 < \alpha_1 < \cdots$ and
$\alpha_k \to \infty$. Moreover, if $\alpha_j=0$ we assume that
$a_{jk}=0$ for $k>0$.
\item $\tr(e^{-t H_1}-e^{-t H_0}) = h + O(e^{-ct})$, as $t\to \infty$, where $h=\dim \ker H_1 -\dim \ker H_0$.
\end{enumerate}
The relative spectral zeta function is defined as:
$$\zeta(s;H_1,H_0)= {\frac{1}{\Gamma(s)}} \int_{0}^{\infty}
(\tr(e^{-t H_1}-e^{-t H_{0}})-h) t^{s-1} dt.$$
Thanks to the properties given above, it has a meromorphic extension to the complex plane that
is meromorphic at $s=0$. The relative
determinant is then defined as:
$$\det(H_1, H_0) := e^{-\zeta'(0;H_1,H_{0})}.$$

This determinant is multiplicative. If the determinant of each operator can be defined se\-pa\-ra\-tely,
then their relative determinant is the quotient of the determinants.
In this paper we work with the relative determinant of the following pairs:
$(\Delta_{g}, \bar{\Delta}_{a,0})$, $(\Delta_{g},\Delta_{Z_{a},D})$. The good definition of these relative
determinants is guaranteed by the results of W. M\"uller in \cite{Mu} and \cite{Mu3}.

\section{Splitting formula}
\label{section:splitting} Splitting formulas for determinants have
been widely studied. They have been proved in the setting of compact
manifold by D. Burghelea, L. Friedlander and T. Kappeler in \cite{BFK}, and
in other settings by many other authors. For example, for manifolds with
cylindrical ends they were studied by J. M\"uller and
W. M\"uller in \cite{MuMu} and Loya and Park in \cite{LoyaPark}. In this section
we use the Dirichlet-to-Neumann operator
for the Laplacian on a manifold with cusps to obtain a splitting
formula for the relative determinant
$\det(\Delta_{g},\bar{\Delta}_{\beta,0})$.

\subsection{Dirichlet-to-Neumann operator for $\Delta_{g}$} Let us
start by recalling the definition of the Dirichlet-to-Neumann operator
${\mathcal N}(z)$ and its main properties. Then, we study the limit
operator as the parameter $z$ goes to zero.

Let us assume that $(M,g)$ has only one cusps and that we can decompose it as $M = M_{0} \cup
Z_{\alpha}$ where
$\alpha\geq 1$ and $Z_{\alpha}$ is isometric to
$[\alpha,\infty)\times S^{1}$ with the hyperbolic metric.

Let $\beta \geq \alpha$, then $M$ may be decomposed as $M = M_{\beta}\cup Z_{\beta}$,
with $M_{\beta}=M_{0}\cup [\alpha,\beta]\times S^{1}$, $Z_{\beta}=[\beta,\infty)\times S^{1}$, and
$\Sigma_{\beta}= \{\beta\}\times S^{1}=\partial M_{\beta} = \partial Z_{\beta}$. Let
$\Delta_{M_{\beta}}$ denote the Laplace operator acting on $C^{\infty}(M_{\beta})$ and $\Delta_{M_{\beta},D}$
denote its self-adjoint extension with respect to Dirichlet boundary conditions at $\Sigma_{\beta}$.
Let $\Delta_{Z_{\beta},D}$ be as it was defined in Section \ref{subsection:swc}. We will
explicitly compute the part of the Dirichlet-to-Neumann operator
${\mathcal N}(z)$ on $\Sigma_{\beta}$, for
any value of $\beta > \alpha$ coming from the cusps $Z_{\beta}$. The metric on $\Sigma_{\beta}$ is given by
$g_{\Sigma_{\beta}}= \beta^{-2} dx^{2}$, the eigenvalues of
the Laplacian $\Delta_{\Sigma_{\beta}}$ are $\{4 \pi^{2} n^{2}
\beta^{2}\}_{n \in \Z}$ and the corresponding eigenfunctions are
$\{\beta \exp{(2\pi i n x)}\}_{n\in \Z}$.

Let $z$ be in the resolvent set of $\Delta_{g}$, $\rho(\Delta_{g})$.
Then the Dirichlet-to-Neumann operator,
$${\mathcal N}(z): C^{\infty}(\Sigma_{\beta}) \to C^{\infty}(\Sigma_{\beta}),$$ is
defined as follows: Let $f\in C^{\infty}(\Sigma_{\beta})$ and let
$\tilde{f}$ be the unique square integrable solution to the problem
$$\left\{
\begin{array}{ll}
    (\Delta_{g}-z)\tilde{f}=0 & \mbox{ in } M\setminus \Sigma_{\beta}\\
    \tilde{f}=f & \mbox{ on } \Sigma_{\beta}.\\
\end{array}%
\right.$$ Let $n^{+}$ denote the inwards unit normal vector field at
$\Sigma_{\beta}$ on $M_{\beta}$ and $n^{-}$ the one on
$Z_{\beta}$. Then ${\mathcal N}(z)f$ is defined by the following
equation
$${\mathcal N}(z)f := -\left({\frac{\partial}{\partial
n^{+}}}\left(\tilde{f}\left\vert_{M_{\beta}} \right.\right) +
{\frac{\partial}{\partial n^{-}}} \left(\tilde{f}\left\vert_{Z_{\beta}}
\right. \right)  \right).$$

Theorem $2.1$ of G. Carron in \cite{Caron} establishes that for $z\in \C
\setminus [0,\infty)$, the Dirichlet-to-Neumann operator is a
$1$st-order elliptic, invertible, pseudo\-di\-ffe\-rential operator whose
principal symbol is a scalar, $\text{sym}_{p}({\mathcal N}(z))(x,\eta)=
2 \sqrt{g_{x}(\eta,\eta)}$, $(x,\eta)\in T^{*}M$. In addition, the
function $z\mapsto {\mathcal N}(z)$ is
holomorphic as function of $z$. In particular, ${\mathcal N}(z)$ $:
C^{\infty}(\Sigma_{\beta}) \to C^{\infty}(\Sigma_{\beta})$
has continuous extensions to the Sobolev spaces,
$H^{1}(\Sigma_{\beta})\to L^{2}(\Sigma_{\beta})\to
H^{-1}(\Sigma_{\beta})$.
Then we can think of ${\mathcal N}(z)$ as an operator on $L^{2}(\Sigma_{\beta})$ by
${\mathcal N}(z):H^{1}(\Sigma_{\beta})\subset L^{2}(\Sigma_{\beta})\to
L^{2}(\Sigma_{\beta})$.
Furthermore, for $f\in C^{\infty}(\Sigma_{\beta})$ we have that:
\begin{equation} {\mathcal N}(z)^{-1}f (x) = \int_{\Sigma_{\beta}} G(x,y,z) f(y)
d\mu(y),\label{eq:SkinvDtNopC}\end{equation} where $G(x,y,z)$ is the
Schwartz kernel of $(\Delta_{g} - z)^{-1}$ on $M$, see Theorem 2.1
in \cite{Caron}. This expression is equivalent to:
\begin{equation}{\mathcal N}(z)^{-1}f = \rho_{\Sigma_{\beta}} \circ (\Delta_{g} - z)^{-1} \circ
i_{\Sigma_{\beta}} (f),\label{eq:DtNotRestResIncl}\end{equation}
\noindent where $\rho_{\Sigma_{\beta}}$ denotes the restriction to $\Sigma_{\beta}$ and
$i_{\Sigma_{\beta}}(f) = f\otimes \delta_{\Sigma_{\beta}}$ in the
distributional sense, this means $f\otimes
\delta_{\Sigma_{\beta}}(\varphi)=\int_{\Sigma_{\beta}}\varphi\cdot
f$ for any $\varphi\in C^{\infty}(M)$.

Now, remember that $0\in \sigma(\Delta_{g})$ is an isolated
eigenvalue. Thus the Dirichlet-to-Neumann ope\-ra\-tor ${\mathcal N}(z)$
is actually defined for $z$ in a neighborhood of zero and it makes
sense to consider its limit as $z$ approaches zero. Indeed,
it exists for $z=0$ and the dependence on $z$  is
continuous. In order to prove this, we split the problem in the
classical way letting ${\mathcal N}(z) = {\mathcal N}_{1}(z) + {\mathcal N}_{2}(z),$ where for $i=1,2$ ${\mathcal N}_{i}(z)$ is defined as
follows:

Let $f\in C^{\infty}(\Sigma_{\beta})$, then let $\varphi_{1} \in
C^{\infty}(M_{\beta}\setminus \Sigma_{\beta}) \cap C^{0}(M_{\beta})$
be the unique solution to the problem
$$\left\{%
\begin{array}{ll}
    (\Delta-z)\varphi_{1} = 0 & \mbox{ in } M_{\beta}\setminus \Sigma_{\beta}\\
    \varphi_{1} = f & \mbox{ on } \Sigma_{\beta}.\\
\end{array}%
\right.$$
Put ${\mathcal N}_{1}(z) f = -{\frac{\partial
\varphi_{1}}{\partial n^{+}}}.$ Similarly, let $\varphi_{2} \in
C^{\infty}(Z_{\beta})\cap L^{2}(Z_{\beta})$ be the unique square
integrable solution to the problem:
$$\left\{%
\begin{array}{ll}
    (\Delta-z)\varphi_{2} = 0 & \mbox{ in } Z_{\beta}\\
    \varphi_{2} = f & \mbox{ on } \Sigma_{\beta}.\\
\end{array}%
\right.$$ Put ${\mathcal N}_{2}(z) f = -{\frac{\partial
\varphi_{2}}{\partial n^{-}}}.$

Using the usual method of separation of variables in the cusp we can
compute the operator ${\mathcal N}_{2}(z)$ explicitly. The explicit
expression of ${\mathcal N}_{2}(z)$ is useful to compute the
limit of the operator as $z\to 0$.

\begin{lemma}
Let $f\in C^{\infty}(\Sigma_{\beta})$. Write $z=s(1-s)$. If $\Re(s)> {\frac{1}{2}}$ then
\begin{equation}
{\mathcal N}_{2}(s(1-s)) f = -(1-2s)
c_{0}(f)\beta - s f + \beta
\sqrt{\Delta_{\Sigma_{\beta}}} \ {\frac{K_{s+{\frac{1}{2}}}}{K_{s-{\frac{1}{2}}}}}\left(\beta
\sqrt{\Delta_{\Sigma_{\beta}}}\right) f, \label{eq:N2sm1/2}
\end{equation}
where $c_{0}(f)$ is the projection of $f$ on the kernel of
$\Delta_{\Sigma_{\beta}}$ and $K_{\nu}$ is the modified Bessel function of order
$\nu$. In the case $\Re(s)< {\frac{1}{2}}$,
\begin{equation}{\mathcal N}_{2}(s)f(x)= -s
f(x) + \beta \sqrt{\Delta_{\Sigma_{\beta}}} \
{\frac{K_{s+{\frac{1}{2}}}}
{K_{s-{\frac{1}{2}}}}}\left(\beta\sqrt{\Delta_{\Sigma_{\beta}}}\right) f(x).
\label{eq:N2ss1/2}\end{equation}
If $\Re(s)={\frac{1}{2}}$, $f\in \dom({\mathcal N}_{2}(z))$ only if it its zero Fourier coefficient vanishes,
$\int_{\Sigma_{\beta}}f dA_{\Sigma_{\beta}} = 0$. In this
case we have:
\begin{equation}{\mathcal N}_{2}(s) f = - s f +
\beta\sqrt{\Delta_{\Sigma_{\beta}}} \
{\frac{K_{s+{\frac{1}{2}}}}
{K_{s-{\frac{1}{2}}}}}\left(\beta\sqrt{\Delta_{\Sigma_{\beta}}}\right) f.
\label{eq:N2seq1/2}
\end{equation} \label{lemma:DtN2expls}
\end{lemma}

\begin{proof}
Take the Fourier expansion of $\varphi_{2}$ and $f$ on the cusp, $\varphi_{2}(y,x) = \sum_{n\in {\mathbb Z}}
a_{n}(y) \beta e^{2\pi i n x}$ and $f(x) = \sum_{n\in {\mathbb Z}} c_{n} \beta e^{2\pi i n x}$.
Then, using separation of variables the problem becomes
$$\left\{%
\begin{array}{ll}
    (-y^{2}{\frac{d^{2}}{dy^{2}}} + y^{2}4\pi^{2} n^{2}\beta^{2} - z) a_{n}(y) =
0 \\
    a_{n}(\beta) = c_{n}, & \mbox{ for } n\in \Z.\\
\end{array}
\right. $$

Set $z = s(1-s)$ with $s\in \C$. Then for $n\neq 0$, two
linear independent solutions of the equation
\begin{equation} \left( -y^{2} {\frac{d^{2}}{dy^{2}}} + 4\pi^{2} n^{2}\beta^{2}
y^{2} - s(1-s) \right) a_{n}(y) = 0
\label{eq:fcoefsvs}\end{equation}
\noindent are $y^{\frac{1}{2}} K_{s-{\frac{1}{2}}}(2\pi \vert n
\vert \beta y)$ and
 $y^{\frac{1}{2}} I_{s-{\frac{1}{2}}}(2\pi \vert n \vert \beta y)$,
 where $K_{s-\frac{1}{2}}$ and $I_{s-\frac{1}{2}}$ are the modified Bessel
functions.
The function $I_{s-{\frac{1}{2}}}$ is discarded because it is not square integrable
on $[1,\infty)$ for any value of $s$. Thus,
$$\varphi_{2}(y,x)=
b_{0,1}y^{s}\beta +
b_{0,2} y^{1-s}\beta + \sum_{n \neq 0} b_{n} y^{{\frac{1}{2}}}
K_{s-{\frac{1}{2}}}(2\pi \vert n\vert \beta y) \beta e^{2\pi i n
x}.$$ \noindent Then for $n\neq 0$, $a_{n}(y)= b_{n}
y^{{\frac{1}{2}}} K_{s-{\frac{1}{2}}}(2\pi \vert n\vert \beta y)$,
where $b_{n}$ and $b_{0,1}, b_{0,2}$ are constants determined by the
boundary and the square integrable conditions.

{\bf Case $\Re(s)> {\frac{1}{2}}$.}  In this case $b_{0,1}=0$ and
$y^{{\frac{1}{2}}}K_{s-{\frac{1}{2}}}(2\pi \vert n \vert \beta y)$
is square integrable on $[1, \infty[$. Then we have:
$$\varphi_{2}(y,x)= b_{0,2}y^{1-s}\beta + \sum_{n \neq 0} b_{n}
y^{{\frac{1}{2}}}
K_{s-{\frac{1}{2}}}(2\pi \vert n \vert \beta y) \beta e^{2\pi i n
x},$$ where $a_{0}(y) = b_{0,2} y^{1-s}$ and $a_{n}(y)=  b_{n}
y^{{\frac{1}{2}}} K_{s-{\frac{1}{2}}}(2\pi \vert n \vert \beta y)$.
The boundary condition $\varphi_{2}(\beta,x) = f(x)$ is equivalent
to $a_{n}(\beta)=c_{n}$. Thus $b_{0,2} = c_{0}\beta^{s-1}$ and
$$b_{n}= {\frac{c_{n}}{\beta^{{\frac{1}{2}}}K_{s-{\frac{1}{2}}}(2\pi \vert n
\vert \beta^{2})}}.$$
In this way we obtain:
$$\varphi_{2}(y,x)= c_{0}\beta^{s} y^{1-s} + \sum_{n\neq 0}
{\frac{c_{n}}{\beta^{{\frac{1}{2}}} K_{s-{\frac{1}{2}}}
(2\pi\vert n\vert \beta^{2})}} \ y^{{\frac{1}{2}}}
K_{s-{\frac{1}{2}}}(2\pi \vert n \vert \beta y) \beta e^{2\pi i n
x},$$
after differentiation, evaluation at $y=\beta$ gives:
\begin{align*}
y \left. {\frac{\partial}{\partial y}}\ \varphi_{2}
(y,x)\right\vert_{y=\beta} &= (1-s) c_{0}\beta + \beta \sum_{n\neq
0} c_{n} \left( s\beta^{-1} - 2\pi \vert n\vert \beta
{\frac{K_{s+{\frac{1}{2}}}(2\pi \vert
n\vert\beta^{2})}{K_{s-{\frac{1}{2}}}(2\pi \vert n\vert \beta^{2})}}
\right)\beta e^{2\pi i n x}\\
&=(1-2s) c_{0} \beta + s f(x) - \beta
\sqrt{\Delta_{\Sigma_{\beta}}}\
{\frac{K_{s+{\frac{1}{2}}}}{K_{s-{\frac{1}{2}}}}}
(\beta\sqrt{\Delta_{\Sigma_{\beta}}}) f(x),
\end{align*}
where we have chosen the positive square root of the eigenvalues to define the operator $\sqrt{\Delta_{\Sigma_{\beta}}}$.

{\bf Case $\Re(s) = {\frac{1}{2}}$.} The computations are the same as in the previous case but the
square integrability condition implies that the zero term in the
Fourier expansion of the solution $\varphi_{2}$ should be null, thus
$$\varphi_{2}(y,x) = \sum_{n \neq 0} b_{n} y^{{\frac{1}{2}}}
K_{s-{\frac{1}{2}}}(2\pi \vert n \vert \beta y)\beta e^{2\pi i n
x}.$$ In addition, the condition $a_{0} = c_{0}$ gives $c_{0}=0$. This means
that only in the case when $c_{0}=0$ will there exist a solution to
the problem. Hence for $f$ to be in the domain of ${\mathcal N}_{2}(s(1-1))$, $f$ should satisfy $c_{0}(f) =
\int_{\Sigma_{\beta}}f dA_{\Sigma_{\beta}} = 0$. For such functions
$f$ equation (\ref{eq:N2seq1/2}) holds.

The case when $\Re(s) < {\frac{1}{2}}$ is similar.
\end{proof}

\begin{remark}
Let $z<0$, then the operator ${\mathcal N}(z)$ is positive. This follows from
the non-negativity of the Laplacian $\Delta_{g}$ and the definition of ${\mathcal N}(z)$. Remember that the
Schwartz kernel of ${\mathcal N}(z)^{-1}$ is the same as the Schwartz
kernel of $(\Delta_{g}-z)^{-1}$. We have $\Delta_{g}\geq 0$. If
$z<0$, then $(\Delta_{g}-z)>0$, and $(\Delta_{g}-z)^{-1}>0$.
Therefore ${\mathcal N}(z)^{-1}>0$. In addition, in this case ${\mathcal N}(z)$ is also self-adjoint and by
the work of Kontsevich and Vishik in \cite{KoVi} we know that its zeta determinant
is well defined.
\end{remark}

The existence of the Dirichlet-to-Neumann operator when $z=0$ is given by the following lemma:
\begin{lemma}
For every $f\in C^{\infty}(\Sigma_{\beta})$ there exists a unique
solution $\tilde{f} \in C^{\infty}(M\setminus \Sigma_{\beta})\cap C^{0}(M)$ and $\tilde{f}\vert_{Z_\beta} \in L^{2}$, to the problem:
$$\left\{%
\begin{array}{ll}
    \Delta_{g}\tilde{f} =0 & \mbox{ in } M\setminus \Sigma_{\beta}\\
    \tilde{f}=f & \mbox{ on } \Sigma_{\beta}.\\
\end{array}%
\right.$$ In addition, using the notation introduced above we have that:
\begin{equation*}{\mathcal N}_{2} f := -\left. y{\frac{\partial}{\partial y}} \varphi_{2}(y,x)
\right\vert_{y=\beta}
= \beta\sqrt{\Delta_{\Sigma_{\beta}}} f.\end{equation*}
\label{lemma:DtNm0cusp}\end{lemma}

\begin{proof}
In the same way as in the proof of Lemma 3.1 in J. M\"uller and W. M\"uller \cite{MuMu}, the uniqueness of the
solution $\varphi_{1}\in C^{\infty}(M_{\beta}\setminus
\Sigma_{\beta}) \cap C^{0}(M_{\beta})$ of the Dirichlet problem on
$M_{\beta}$ follows from the invertibility of
$\Delta_{M_{\beta},D}$. The uniqueness of the solution on
$Z_{\beta}$ also follows from the invertibility of
$\Delta_{Z_{\beta},D}$. To see the existence on
$Z_{\beta}$ more explicitly let us follow the same procedure
of the proof of Lemma \ref{lemma:DtN2expls} but taking $z=0$. One way to obtain $z=0$ is to take $s=1$ in
equation (\ref{eq:fcoefsvs}). In this case the square integrable
condition gives
$$\varphi_{2}(y,x) = \sum_{n\in {\mathbb Z}} a_{n}(y) e^{2\pi i n x} =
b_{0,2}\beta + \sum_{n\neq 0} b_{n}
y^{\frac{1}{2}} K_{\frac{1}{2}}(2\pi \vert n \vert \beta y) \beta
e^{2\pi i n x}.$$ We know that $K_{\frac{1}{2}}(r)=
\sqrt{\frac{\pi}{2}} r^{-\frac{1}{2}} e^{-r}.$ Then for $n\neq 0$ we
have $a_{n}(y) = {\frac{b_{n}}{2\sqrt{\vert n\vert \beta}}}\
e^{-2\pi \vert n \vert \beta  y}.$ The boundary condition
$\varphi_{2}(\beta,x) = f(x)$, which is equivalent to
$a_{n}(\beta)=c_{n}$, gives $b_{0} = c_{0}$ and  $b_{n} = c_{n}
2\sqrt{\vert n\vert \beta} e^{2\pi \vert n\vert \beta^{2}}$. Then
$$\varphi_{2}(y,x)= c_{0}\beta + \sum_{n\neq 0} c_{n} e^{2\pi \vert
n\vert \beta^{2}} e^{-2\pi \vert n \vert \beta y} \beta e^{2\pi i n
x}.$$ Taking the inward derivative we obtain:
$$\left. y{\frac{\partial}{\partial y}} \varphi_{2}(y,x) \right\vert_{y=\beta} =
\beta\sum_{n\neq 0} -2\pi \vert n\vert \beta c_{n}\ \beta e^{2\pi i n x} =
-\beta \sqrt{\Delta_{\Sigma_{\beta}}} f.$$
The other way to obtain $z=0$ is taking $s=0$ in equation (\ref{eq:fcoefsvs}).
In this case we have:
\begin{multline*}\varphi_{2}(y,x) =
b_{0,1}\beta + \sum_{n\neq 0} b_{n}
y^{\frac{1}{2}} K_{-\frac{1}{2}}(2\pi \vert n \vert \beta y) \beta
e^{2\pi i n x}\\ = c_{0}\beta + \sum_{n\neq 0} c_{n} e^{2\pi \vert n\vert
\beta^{2}}
e^{-2\pi \vert n \vert \beta y} \beta e^{2\pi i n x},\end{multline*}
where we have used that $K_{-\frac{1}{2}}=
K_{\frac{1}{2}}$. Thus for
$s=0$ and for $s=1$, the solutions of the Dirichlet
problem on $Z_{\beta}$ are the same. Since
$\varphi_{1}\vert_{\Sigma_{\beta}}=\varphi_{2}\vert_{\Sigma_{\beta}}$,
we have that the solution $\tilde{f}$ is continuous on $M$.
\end{proof}

\begin{remark}
For $z\in \rho(\Delta_{g})$, the resolvent set of $\Delta_{g}$, it
is well known that ${\mathcal N}_{1}(z)$ is a $1$st order invertible
elliptic pseudodifferential operator. The limit, ${\mathcal N}_{1}$, as
$z\to 0$, it is well known to be a $1$st order elliptic
pseudodifferential operator, but it is non-invertible, see for
example D. Burghelea, L. Friedlander and T. Kappeler in \cite{BFK} and M.E. Taylor in \cite{Tay2} section $7.11$. Therefore the
ope\-ra\-tor ${\mathcal N}= {\mathcal N}_{1} + {\mathcal N}_{2}$ is
non-invertible. However it is non-negative and $\dim(\ker({\mathcal N}))=1$.
\end{remark}

\begin{proposition}
Let $f\in C^{\infty}(\Sigma_{\beta})$. Then ${\mathcal N}(z)f$ depends
continuously of $z$ in a small enough neighborhood of $z=0$, and
$$\lim_{z\to 0} {\mathcal N}(z)f = {\mathcal N}f.$$ \label{prop:limitDtNo}
\end{proposition}
\begin{proof}
The proof of $\lim_{z \to 0}{\mathcal N}_{1}(z)f= {\mathcal N}_{1}f$ is
the same as the proof of Lemma 3.3 in J. M\"uller and W. M\"uller \cite{MuMu}. For the convenience of the
reader we repeat here the argument with our notation. For $f\in
C^{\infty}(\Sigma_{\beta})$, let $\varphi_{1}(z)$ be the unique
function in $C^{\infty}(M_{\beta}\setminus \Sigma_{\beta})$
satisfying $(\Delta_{g}-z)\varphi_{1}(z)=0$,
$\varphi_{1}(z)\vert_{\Sigma_{\beta}}=f$ and
$$\varphi_{1}(z)=
\tilde{f}-(\Delta_{M_{\beta},D}-z)^{-1}((\Delta_{M_{\beta}}-z)(\tilde{f})),$$
where $\tilde{f}\in C^{\infty}(M_{\beta})$ is any extension of $f$.
Since $\Delta_{M_{\beta},D}$ is invertible, the formula also holds
for $z=0$. From this representation of $\varphi_{1}(z)$, it follows
immediately that ${\mathcal N}_{1}(z)f$ converges to ${\mathcal N}_{1}f$ as
$z\to 0$.

Now let us take the limit of ${\mathcal N}_{2}(z)$ as $s\to 1$. To do that we use
equation (\ref{eq:N2sm1/2})
to obtain:
$$\lim_{s\to 1}{\mathcal N}_{2}(s(1-s)) f = c_{0}\beta - f + \beta
\sqrt{\Delta_{\Sigma_{\beta}}} \
{\frac{K_{{\frac{3}{2}}}}
{K_{{\frac{1}{2}}}}}(\sqrt{\Delta_{\Sigma_{\beta}}}) f.$$
Using the expression
$K_{\frac{3}{2}}(u) = \sqrt{\frac{\pi}{2}} u^{-3/2}e^{-u} (u + 1),$
we have that ${\frac{K_{{\frac{3}{2}}}(2\pi \vert n\vert
\beta^{2})}{K_{{\frac{1}{2}}}(2\pi \vert n\vert\beta^{2})}}=
{\frac{2\pi \vert n\vert\beta^{2} + 1}{2\pi \vert
n\vert\beta^{2}}}.$ Thus,
$$\lim_{s\to 1}{\mathcal N}_{2}(s(1-s)) f
= \beta \sum_{n\neq 0} 2\pi \vert n\vert \beta c_{n}\beta e^{2\pi i n x}=
\beta \sqrt{\Delta_{\Sigma_{\beta}}}f = {\mathcal N}_{2}f.$$

For the limit when $s\to 0$ we have:
\begin{align*}\lim_{s\to 0}{\mathcal N}_{2}(s(1-s)) f &= \lim_{s\to 0} -s f(x)
+ \beta \sum_{n\neq 0} 2\pi \vert n\vert \beta {\frac{K_{s+{\frac{1}{2}}}(2\pi
\vert n\vert \beta^{2})} {K_{s-{\frac{1}{2}}}(2\pi
\vert n\vert \beta^{2})}} \ c_{n} \beta e^{2\pi i n x}\\
&= \beta \sum_{n\neq 0} 2\pi \vert n\vert \beta\ c_{n} \beta e^{2\pi
i n x}= \beta \sqrt{\Delta_{\Sigma_{\beta}}}f .\end{align*} Thus it
follows that
\begin{equation*}
\lim_{s\to 1}{\mathcal N}_{2}(s(1-s))f = \lim_{s\to 0}{\mathcal N}_{2}(s(1-s))f = {\mathcal N}_{2}(0) f = \beta
\sqrt{\Delta_{\Sigma_{\beta}}}f = {\mathcal N}_{2}f.
\end{equation*}
\end{proof}

\subsection{Splitting formula for the relative determinant}

We want to have a splitting formula for the relative determinant
that relates $\det(\Delta_{g},\Delta_{\beta,0})$ to the regularized
determinant of the operator ${\mathcal N}$. We
will use this formula in section \ref{section:compactness} to prove
Theorem \ref{theorem:compactnessisospectral}. For $z\in
\rho(\Delta_{g})$ Corollary $4.6$ in G. Carron \cite{Caron} establishes the
following splitting formula for complete surfaces, which we rewrite
using his notation:
\begin{equation} \det({\mathcal L} - z, {\mathcal L}_{0,D} - z) = \det{\mathcal N}(z), \label{eq:splitformc}\end{equation}
where ${\mathcal L}$ is the
self-adjoint extension of the Laplacian on $M$ and ${\mathcal L}_{0,D}$
is the self-adjoint extension of the Laplacian on $M\setminus
\Sigma$ with Dirichlet boundary conditions on $\Sigma$. Let $\lambda > 0$,
 put $z=-\lambda$ and let us denote ${\mathcal N}(-\lambda)$ by
$R(\lambda)$. Then $R(\lambda)>0$ and it has the same properties as
${\mathcal N}(-\lambda)$. In our
case equation (\ref{eq:splitformc}) has the form:
\begin{equation}\det(\Delta_{g} + \lambda, \Delta_{Z_{\beta,D}} + \lambda)
(\det(\Delta_{M_{\beta,D}} + \lambda))^{-1} = \det {\mathcal N}(-\lambda)
= \det R(\lambda).
\label{eq:splitformournot}\end{equation}
Both sides of equation (\ref{eq:splitformournot}) diverge as
$\lambda \to 0^{+}$, we study how is this divergence.

\begin{lemma} As $\lambda \to 0^{+}$, the left hand side of equation (\ref{eq:splitformournot}) has the
following decomposition:
\begin{multline*}\log \det (\Delta_{g} +\lambda, \Delta_{Z_{\beta},D}+\lambda) - \log
\det(\Delta_{M_{\beta,D}} + \lambda)\\
 = \log\lambda +
\log \det(\Delta_{g}, \Delta_{Z_{\beta},D}) - \log \det\Delta_{M_{\beta,D}} +
o(1).\end{multline*}
\end{lemma}

\begin{proof} Let us go back to the definition of the relative determinant and
use the definition of the relative zeta functions for $(\Delta_{g}, \Delta_{Z_{\beta},D})$ and
$(\Delta_{g}+ \lambda, \Delta_{Z_{\beta},D} + \lambda)$. From them we have:
\begin{align*}
&\zeta(s;\Delta_{g}+ \lambda, \Delta_{Z_{\beta},D} + \lambda) =
{\frac{1}{\Gamma(s)}}\int_{0}^{\infty} \tr(e^{-t\Delta_{g}} -
e^{-t\Delta_{Z_{\beta},D}}) e^{-t\lambda} t^{s-1}dt\\
\quad &= {\frac{1}{\Gamma(s)}}\int_{0}^{\infty} (\tr(e^{-t\Delta_{g}} -
e^{-t\Delta_{Z_{\beta},D}})-1) e^{-t\lambda} t^{s-1}dt +
{\frac{1}{\Gamma(s)}}
\Gamma(s) \lambda^{-s}\\
\quad &= \lambda^{-s} + \zeta(s,\Delta_{g},\Delta_{Z_{\beta},D}) +
{\frac{\lambda}{\Gamma(s)}} \int_{0}^{\infty} (\tr(e^{-t\Delta_{g}} -
e^{-t\Delta_{Z_{\beta},D}})-1) {\frac{e^{-t\lambda}-1}{\lambda}}
t^{s-1} dt.
\end{align*}
The last integral converges in a half plane. Therefore due to the
asymptotic expansions of the relative heat trace for small and large $t$,
it has an analytic continuation that is holomorphic at $s=0$. So, as $\lambda \to 0^{+}$ we obtain:
\begin{align*}
\left. {\frac{d}{ds}} \zeta(s;\Delta_{g}+ \lambda,
\Delta_{Z_{\beta},D} + \lambda)\right\vert_{s=0} &= -\log \lambda +
\left. {\frac{d}{ds}} \zeta(s;\Delta_{g},
\Delta_{Z_{\beta},D})\right\vert_{s=0}-o(1),\\
-\log \det(\Delta_{g}+ \lambda, \Delta_{Z_{\beta},D} + \lambda) &=
-\log \lambda - \log \det(\Delta_{g}, \Delta_{Z_{\beta},D}) -o(1),
\end{align*}
as desired. Similarly, from the definition of $\zeta_{\Delta_{M_{\beta,D}} + \lambda}
(s)$ it follows that $$\log \det(\Delta_{M_{\beta,D}} + \lambda) =
\log \det(\Delta_{M_{\beta,D}}) + o(1),$$ as $\lambda \to 0^{+}$. This
finishes the proof of the lemma.
\end{proof}

In order to study the asymptotic behavior of the right hand side of equation
(\ref{eq:splitformournot}), we need some preliminaries.

Let $\lambda > 0$ and $R(\lambda)$ be as above. Recall that $R\geq 0$, $\ker R =\C$ and
$\lim_{\lambda \to 0}R(\lambda)=R$. It is not difficult to prove that $R$ is self-adjoint; therefore
the regularized determinant of $R$, $\det^* R$, may defined by the
meromorphic continuation of $$\zeta^*_{R}(s)=\sum_{\mu_{i}>0}
\mu_{i}^{-s},$$
where the sum runs over the positive eigenvalues of $R$.

Now, let $\mu_{1}$ be the first non-zero eigenvalue of $R$, let $0<\mu < \mu_{1}$, and
let $P_{\mu}$ be the spectral projection of the Laplacian $\Delta_{g}$ on
$[0,\mu]$. Then by equation (\ref{eq:DtNotRestResIncl}), $R(\lambda)^{-1}$ can be decomposed as:
$$R(\lambda)^{-1} = \rho_{\Sigma_{\beta}} \circ P_{\mu}(\Delta_{g} +
\lambda)^{-1} \circ i_{\Sigma_{\beta}}
+ \rho_{\Sigma_{\beta}} \circ (I-P_{\mu})(\Delta_{g} + \lambda)^{-1} \circ
i_{\Sigma_{\beta}} =: Q_{\mu}(\lambda) +
\wt{Q}_{\mu}(\lambda).$$
The kernel of $Q_{\mu}(\lambda)$ in terms of the spectral decomposition of $\Delta_{g}$ on $M$
is given by:
\begin{multline*}K_{Q_{\mu}(\lambda)}(x,x',\lambda) = \sum_{0\leq \lambda_{j}\leq \mu}
{\frac{1}{\lambda_{j}+\lambda}}
\varphi_{j}(x) \overline{\varphi_{j}(x')}\\ + {\frac{1}{2\pi}}
\int_{0}^{\mu} {\frac{1}{\lambda + 1/4 + r^{2}}}
E(x,{\frac{1}{2}}+ir) E(x',{\frac{1}{2}}-ir) dr,\end{multline*}
for $x,x'\in \Sigma_{\beta}$.
We can further decompose $Q_{\mu}(\lambda)$ as $Q_{\mu,1}(\lambda) +
Q_{\mu,2}(\lambda)$, where $Q_{\mu,1}(\lambda)$ is given by:
$$Q_{\mu,1}(\lambda)f = {\frac{1}{\lambda}} {\frac{1}{A_{g}}}
\int_{\Sigma_{\beta}} f(x) d\mu(x), \ \mbox{ with }\
K_{Q_{\mu,1}(\lambda)}(x,x',\lambda) = {\frac{1}{\lambda}}
{\frac{1}{A_{g}}},$$ and $K_{Q_{\mu,2}(\lambda)} = K_{Q_{\mu}(\lambda)} - K_{Q_{\mu,1}(\lambda)}$.
Taking the limit as $\lambda\to 0$ of $K_{Q_{\mu,2}(\lambda)}$ we obtain:
\begin{multline*} \lim_{\lambda \to 0} K_{Q_{\mu,2}(\lambda)}(x,x',\lambda)= \sum_{0<
\lambda_{j}\leq \mu} {\frac{1}{\lambda_{j}}}
\varphi_{j}(x) \overline{\varphi_{j}(x')}\\ + {\frac{1}{2\pi}}
\int_{0}^{\mu} {\frac{1}{1/4 + r^{2}}} E(x,{\frac{1}{2}}+ir)
E(x',{\frac{1}{2}}-ir)dr.\end{multline*} Thus $\Vert Q_{\mu,2}(\lambda) \Vert$ remains bounded as $\lambda \to 0^{+}$.
In the same way as in \cite[Lemma 3.5]{MuMu}, we can prove that
there is a constant $C>0$, depending only on $\mu$, such that for all $\lambda > 0$:
\begin{equation}\Vert \wt{Q}_{\mu}(\lambda) \Vert =
\Vert
\rho_{\Sigma_{\beta}}\ \circ\  (I-P_{\mu})(\Delta_{g}+ \lambda)^{-1}
\circ i_{\Sigma_{\beta}}\Vert \leq C.\label{eq:ubcomresDtNaux}
\end{equation} 

Therefore the operator $R(\lambda)^{-1}$ can
be written as:
\begin{equation}
R(\lambda)^{-1} = Q_{\mu,1}(\lambda) + K_{\mu}(\lambda),\label{eq:decIDtNl}\end{equation}
with $\Vert K_{\mu}(\lambda)\Vert$ uniformly bounded as $\lambda \to 0^{+}$.

Now, just note that $\rho_{\Sigma_\beta}(\ker(\Delta_g)) = \ker(R)$. Let $0< \mu_{1}(\lambda) \leq \mu_{2}(\lambda) \leq \mu_{3}(\lambda) \leq
\dots $ be the eigenvalues of $R(\lambda)$. Then from the discussion above it is clear that:
\begin{eqnarray*}
\mu_{1}(\lambda)\to 0,& & \mbox{ as } \lambda\to 0,\\
\mu_{i}(\lambda) \geq c >0, & & \mbox{ for }i\geq2, \ \lambda\geq 0.
\end{eqnarray*}

\begin{lemma}
There is the following asymptotic expansion as $\lambda\to 0^{+}$:
\begin{equation}
\log \det R(\lambda) = \log \mu_{1}(\lambda) + \log {\det}^* R +
o(1). \label{eq:fdecfdetDtN}\end{equation}
\end{lemma}
\begin{proof}
Let $\ker(R)$ be the kernel of $R$, ${\mathcal H}=(\ker(R))^{\perp}$ be its
orthogonal complement, and
$P:L^{2}(\Sigma_{\beta})\to \ker(R)$ and
$P^{\perp}:L^{2}(\Sigma_{\beta})\to {\mathcal H}$ be the corresponding
orthogonal projections. By definition:
$$\log \det R(\lambda) :=
-\left.{\frac{d}{ds}}\right\vert_{s=0}\zeta_{R(\lambda)}(s)
=  -\left.{\frac{d}{ds}}\right\vert_{s=0} {\frac{1}{\Gamma(s)}}
\int_{0}^{\infty} \tr(e^{-tR(\lambda)}) t^{s-1} dt. $$ The first
thing to do is to separate the first eigenvalue. For that, let
$\gamma$ be a contour in $\C$ contained in the resolvent set of $R(\lambda)$, $\rho(R(\lambda))$, and
surrounding the spectrum of $R(\lambda)$, for all $\lambda\geq 0$
small enough. Then:
\begin{align*}
e^{-tR(\lambda)} &= {\frac{1}{2i \pi}} \int_{\gamma}
e^{-t\xi} (R(\lambda)-\xi)^{-1} d\xi\\ &=  {\frac{1}{2i \pi}} \int_{\gamma_{1}}
e^{-t\xi} (R(\lambda)-\xi)^{-1} d\xi +  {\frac{1}{2i \pi}} \int_{\gamma_{2}}
e^{-t\xi} (R(\lambda)-\xi)^{-1} d\xi,
\end{align*}
where $\gamma_{1}$ is a contour surrounding $\{\mu_{1}(\lambda),0\}$
and $\gamma_{2}$ surrounds the half line $[c,\infty)$, where
$\mu_{2}(\lambda)\geq c$ for all $\lambda >0$. The curves $\gamma_{1}$ and $\gamma_{2}$ can be chosen without
overlapping and independently of $\lambda$. It is clear that:
$${\frac{1}{2i \pi}} \int_{\gamma_{1}}
e^{-t\xi} (R(\lambda)-\xi)^{-1} d\xi =
e^{-t\mu_{1}(\lambda)}P(\lambda),$$
where $P(\lambda)$ is the orthogonal projection on the $\mu_{1}(\lambda)$-eigenspace. Therefore
\begin{multline*}\zeta_{R(\lambda)}(s) = {\frac{1}{\Gamma(s)}} \int_{0}^{\infty}
e^{-t \mu_{1}(\lambda)} t^{s-1} dt\\ + {\frac{1}{\Gamma(s)}}
\int_{0}^{\infty} \tr\left({\frac{1}{2i \pi}}\int_{\gamma_{2}} e^{-t\xi} (R(\lambda)-\xi)^{-1}
d\xi \right) t^{s-1}
dt.\end{multline*}

The family $R(\lambda)$ acting on
a subspace of $L^{2}(\Sigma_{\beta})$ into $L^{2}(\Sigma_{\beta})$ depends continuously on $\lambda$.
The resolvent of
$R(\lambda)$ depends continuously of $\lambda$ too. Since $R$ has
$0$ as eigenvalue, the resolvent $(R-\xi)^{-1}$ has a pole at
$\xi=0$ and can be written as:
$$(R-\xi)^{-1}= -\xi^{-1} P + A(\xi), $$
with $A(\xi)$ a holomorphic operator in $\xi$. On the other hand,
$\mu_{1}(\lambda)>0$ for $\lambda>0$. Therefore $(R(\lambda)-\xi)^{-1}$ is continuous in $\lambda$ close to $0$
and holomorphic in $\xi$ far from $\sigma(R(\lambda))$. When integrating over
$\gamma_{2}$ we are actually
dealing with the operators $P(\lambda)^{\perp}R(\lambda)$ or
$P^{\perp}R(\lambda)$.
From general results about resolvents we have that
$(P(\lambda)^{\perp}R(\lambda)-\xi)^{-1}$ converges continuously to
$(P^{\perp}R-\xi)^{-1}$ as $\lambda\to 0^{+}$,
for $\xi \in \rho(R(\lambda))$. This fact in addition to the following expressions
\begin{align*}
& e^{-t P(\lambda)^{\perp} R(\lambda)}  = {\frac{1}{2i \pi}}
\int_{\gamma_{2}} e^{-t\xi} (P(\lambda)^{\perp}R(\lambda)-\xi)^{-1}
d\xi, \text{ and, }\\
& e^{-t P^{\perp} R} = {\frac{1}{2i \pi}}
\int_{\gamma_{2}} e^{-t\xi} (P^{\perp}R-\xi)^{-1} d\xi,
\end{align*}
imply that $e^{-t P(\lambda)^{\perp}
R(\lambda)}$ converges to $e^{-t P^{\perp} R}$.
Therefore
$\tr(e^{-t P(\lambda)^{\perp} R(\lambda)})$ depends continuously on $\lambda$
and so does the zeta function.
In this way we obtain:
$$\log \det P(\lambda)^{\perp}R(\lambda) = \log \det P^{\perp}R + o(1), \text{
as } \lambda \to 0^{+}.$$
This finishes the proof of equation (\ref{eq:fdecfdetDtN}).
\end{proof}

We consider now the behavior of the term $\log \mu_{1}(\lambda)$ in equation (\ref{eq:fdecfdetDtN}):
\begin{lemma} As $\lambda \to 0^{+}$
\begin{equation}
\log \mu_{1}(\lambda) = \log\lambda +
\log\left({\frac{A_{g}}{\ell_{\beta}}}\right) + o(1), \label{eq:logmu1}
\end{equation}
where $A_{g}= \mbox{area}(M)$ and $\ell_{\beta}=\mbox{length}(\Sigma_{\beta})$.\label{Lemma:decmu1}
\end{lemma}
\begin{proof} First observe that ${\frac{1}{\mu_{1}(\lambda)}} = \Vert R(\lambda)^{-1}\Vert,$
where the norm is the operator norm in $L^{2}$.
From equations (\ref{eq:decIDtNl}) and $\Vert Q_{\mu,1}(\lambda) \Vert =
{\frac{1}{\lambda}}{\frac{\ell_{\beta}}{A_{g}}},$ it follows that:
$$\Vert R(\lambda)^{-1} \Vert = {\frac{\ell_{\beta}}{\lambda A_{g}}} + O(1) = {\frac{1}{\mu_{1}(\lambda)}}.$$
The expansion for the logarithm applied to ${\frac{1}{\mu_{1}(\lambda)}} = {\frac{\ell_{\beta}}{\lambda
A_{g}}} + u(\lambda)$, with $u(\lambda)=O(1)$ implies that:
$$\log \left({\frac{1}{\mu_{1}(\lambda)}}\right) = -\log(\mu_{1}(\lambda)) =
\log\left({\frac{\ell_{\beta}}{A_{g}}}\right) - \log\lambda + O(\lambda) \mbox{
as } \lambda \to 0^{+}.$$
This finishes the proof of equation (\ref{eq:logmu1}).
\end{proof}

Putting everything together we obtain the following splitting formula:

\begin{theorem} For the relative determinant of the Laplace operator on a surface
with cusps $(M,g)$, and the regularized
determinant of the Dirichlet-to-Neumann operator $R$ on
$\Sigma_{\beta}=\{\beta\}\times S^{1}\subset M$, we have the
following splitting formula:
$${\frac{\det(\Delta_{g}, \Delta_{Z_{\beta},D})}{\det(\Delta_{M_{\beta,D}})}} =
{\frac{A_{g}}{\ell_{\beta}}}{\det}^{*} R,$$
\label{theorem:splitformdets} where $A_{g}$ denotes the area of $M$
and $\ell_{\beta}$ denoted the length of $\Sigma_{\beta}$.
\end{theorem}

\begin{proof}
We start with the splitting formula for $\lambda>0$, and $\lambda
\in \rho(\Delta_{g})$:
$$\log \det(\Delta_{g} + \lambda, \Delta_{Z_{\beta},D} + \lambda) - \log
\det(\Delta_{M_{\beta,D}} + \lambda)
= \log\det R(\lambda)$$ From the previous lemmas we have that:
\begin{multline*}
\log \det(\Delta_{g}, \Delta_{Z_{\beta},D}) + \log\lambda  - \log
\det(\Delta_{M_{\beta},D} + \lambda) + o(1)\\ = \log\mu_{1}(\lambda)
+\log {\det}^* R + o(1)
= \log \lambda + \log\left({\frac{A_{g}}{\ell_{\beta}}}\right) +
\log {\det}^* R + o(1) \end{multline*}
Letting $\lambda \to 0$, we finally obtain:
$$\log \det(\Delta_{g}, \Delta_{Z_{\beta},D}) - \log \det(\Delta_{M_{\beta,D}})
= \log\left({\frac{A_{g}}{\ell_{\beta}}}\right) +
\log {\det}^{*} R.$$
\end{proof}

\begin{remark}
If we further decompose the operator $\Delta_{Z_{\beta},D}$ as
$\Delta_{\beta,0}\oplus \Delta_{Z_{\beta,1}}$ we obtain:
\begin{multline}
\log \det(\Delta_{g}, \Delta_{\beta,0}) -\log
\det(\Delta_{Z_{\beta},1}) - \log \det(\Delta_{M_{\beta,D}})\\ =
\log\left({\frac{A_{g}}{\ell_{\beta}}}\right) + \log {\det}^{*}
R.\label{eq:splitformallop}
\end{multline}
\end{remark}

\section{Compactness}
\label{section:compactness}

In this section we use results of W. M\"uller in \cite{Mu2} and of OPS in \cite{OPS2}. We refer the reader to these references
for all the details.

In \cite{OPS2} OPS proved that sets of isospectral isometry classes of metrics
on closed surfaces are sequentially compact in the $C^{\infty}$-topology. Let us recall some of the
main steps of the proof in the case $\chi(M)<0$. In
that setting, two metrics $g_{1}$ and $g_{2}$ are called isospectral
if the spectra of the Laplacians $\Delta_{g_1}$ and $\Delta_{g_2}$
are the same including multiplicities. In particular, the
regularized determinant and the heat invariants have the same values at $g_{1}$ and $g_{2}$.

To define the notion of convergence they fix a background metric
$g_{0}$. Associated to $g_{0}$, there is the Levi-Civita connection
and the covariant derivative that allow to differentiate in the
whole tensor algebra. A sequence of metrics $\{g_{n}\}_{n\in \N}$
converges to a metric $g$ in $C^{k}$ if $\Vert g_{n}-g \Vert_{C^{k}}
\to 0$, as $n\to \infty$. A sequence of isometry classes of metrics
$\hat{g}_{n}$ converges to an isometry class $\hat{g}$ if there are
re\-pre\-sen\-tatives $h_{n}\in \hat{g}_{n}$, $h\in \hat{g}$, such
that $h_{n} \to h$, as $n\to \infty$. Now, let $\{\rho_{n}\}_{n\in
\N}$ be a sequence of functions in $C^{k}(M)$ and let $\sigma$ be a
fixed metric on $M$. Then $\rho_{n}\sigma \to \rho \sigma$ in
$C^{k}$ as metrics if and only if $\rho_{n}\to \rho$ in $C^{k}$ as
functions. Moreover, if the metrics $\sigma_{n} \to \sigma$ in
$C^{\infty}$, and the function $\rho_{n}\to \rho$ in $C^{k}$, then
the metrics $\rho_{n}\sigma_{n} \to \rho \sigma$ in $C^{k}$.

After defining convergence and isospectrality, OPS consider a
sequence of isospectral iso\-me\-try classes of metrics
$\{\hat{g}_{n}\}_{n\in \N}$ and pick representatives $g_{n}$. For
each $g_{n}$ there is a metric of constant curvature $\tau_{n}$ such
that $g_{n}=e^{2\varphi_n}\tau_{n}$. In this way, they associate to
each $\hat{g}_{n}$ a hyperbolic isometry class $\hat{\tau}_{n}$.
They use that for each $n$, $\det \Delta_{\hat{\tau}_{n}}\geq \det
\Delta_{\hat{g}_{n}}=\text{constant}>0$ and Mumford's compactness
theorem to prove that there exists a subsequence of
$\{\hat{\tau}_{n}\}_{n\in \N}$ that converges to an element
$\hat{\tau}$ in the moduli space. To have compactness of the
conformal factors $\{\varphi_{n}\}_{n\in \N}$, they prove that for
each $k\in \N$ the $k$-th Sobolev norms $\Vert \varphi_{n}\Vert_{k}$
are uniformly bounded. Compactness in the $C^{\infty}$-topology
follows then from Rellich's Lemma and the Sobolev embedding theorems
on $M$. The constant value of the determinant and Polyakov's formula for the
regularized determinant (see \cite{OPS1}) are used to prove
uniform boundedness of the first Sobolev norm. For the higher
Sobolev norms, they use the constant values of the heat invariants.

If we restrict the proof of OPS to a conformal class, we only need the cons\-tant values of the determinant
of the Laplacians and of the heat invariants associated to the metrics.

Now, let $(M,g)$ be a surface of fixed genus $p$ and fixed number
of cusps $m$. We usually take $m=1$ to make the proofs simpler but
the statements hold for general $m$. We take $g$ as the background
Riemannian metric. Let us decompose $M$ as $M = M_{0}
\cup_{\Sigma_{\alpha}} Z_{\alpha}$ where $M_{0}$ is compact with
boundary $\Sigma_{\alpha}$ and the metric on
$Z_{\alpha}=[\alpha,\infty) \times S^{1}$ is the usual hyperbolic
metric.

Let $K$ be a compact subset of $M$ and let us define the
\lq\lq $K$-compactly supported\rq \rq\  conformal class of $g$ as
the set
\begin{equation}[g]_{K} = \{e^{2\varphi}g\ \vert \ \varphi \in
C_{c}^{\infty}(M),\
\supp \varphi \subset K \}. \label{eq:dKsccm}\end{equation}
Since $K$ is compact, there is a $\beta\geq
\alpha$ such that $K\subset M_{\beta}$ and such that $K\cap
\Sigma_{\beta} = \emptyset$. From now on, when $K$ is given we consider $\beta$ fixed.
Then for every metric in
$h\in [g]_{K}$, $(M,h)$ is a surface with cusps and the cusp is contained in $M\setminus M_{\beta}$.

For $s>0$ and $f\in H^{s}(M,g)$, the $s$-Sobolev
norm is given by $\Vert f\Vert_{H^{s}} := \Vert  (\Delta +
I)^{s/2}f\Vert_{L^{2}}$.

Since we restrict to a conformal class the notion of
convergence of metrics reduces to the convergence of the conformal
factors:
\begin{definition} A sequence of metrics $\{g_{n}\}_{n\in \N}$, with
$g_{n}=e^{2\varphi_{n}}g$ converges to a metric
$h$ in $C^{k}$ if and only if the sequence of function
$\{\varphi_{n}\}_{n\in\N}$ converges to a function $\varphi$ in
$C^{k}$.
\end{definition}

As we explained in the introduction, in the setting of surfaces with cusps
the concept of isospectrality is not enough to study inverse problems and it
should be replaced by the concept of isoresonance.
This is motivated by the close relation between eigenvalues and
resonances. The traditional approach to resonances defines them as the poles of certain meromorphic
extension of the resolvent. However, in this paper we rather work with another approach.
We use the definition of the resonance set as it is given in
\cite{Mu2} because our work relies on trace formulae stated there.

Let us recall the precise definition of the resonance set:
In \cite[p.287]{Mu2}, W. M\"uller starts assigning to each $\eta\in \C$ a multiplicity $m(\eta)$:
\begin{enumerate}
\item If $\Re(\eta)\geq 1/2$ and $\eta\neq 1/2$, $m(\eta)$ is the dimension of the eigenspace of $\Delta_{g}$
for the eigenvalue $\eta(1-\eta)$.
\item If $\Re(\eta)<1/2$. Let $E_{\eta(1-\eta)}$ denote the eigenspace of $\Delta_{g}$ for the eigenvalue $\eta(1-\eta)$. If $\phi(s)$
has a pole at $\eta$ of order $n$, then $m(\eta)= \dim(E_{\eta(1-\eta)}) + n$. If $\phi(s)$ has a zero at $\eta$ of order
$\tilde{n}$, then $m(\eta)= \dim(E_{\eta(1-\eta)}) - \tilde{n}$. By the spectral properties
of the Laplacian and the properties of the scattering phase, we know that $m(\eta)\geq 0$.
\item For $\eta={\frac{1}{2}}$ we have that $m({\frac{1}{2}})= {\frac{1}{2}}(\tr(C({\frac{1}{2}}))+m) + 2\dim(E_{{\frac{1}{4}}})$,
where $E_{{\frac{1}{4}}}$ is the ${\frac{1}{4}}$-eigenspace.
\item In any other case, $m(\eta)=0$.
\end{enumerate}

\begin{definition} \cite{Mu2}
The resonance set of $\Delta_{g}$ is the set of all $\eta\in \C$ such that
$m(\eta)>0$. Each element in the set is counted with its multiplicity.
\end{definition}

In this way, the resonance set is the union of the poles and some of the zeros of the
scattering phase $\phi(s)$ in the half-plane $\{s\ \vert \Re(s)<1/2\}$, the set $\{s_{j}\in \C \ \vert s_{j}(1-s_{j}) \text{is
an eigenvalue } of \Delta_{g}\}$ and $\{\frac{1}{2}\}$. Each element carries its multiplicity. In particular, the definition
implies that the resonance set carries the information of the value of $\tr(C({\frac{1}{2}}))$.
Since $C({\frac{1}{2}})$ is a real symmetric matrix with $C({\frac{1}{2}})^{2}=I$, its eigenvalues are $\pm 1$. Then
$\tr(C({\frac{1}{2}}))= 2\ell - m$, where $\ell$ is the algebraic multiplicity of the eigenvalue $+1$. In this way, the resonance
set determines the value of $\ell$ and the algebraic multiplicity of the eigenvalue $-1$. Therefore it determines $\phi({\frac{1}{2}})$.

We are now ready to define isoresonant surfaces with cusps:
\begin{definition}
Two cusp metrics $g_{1}$ and $g_{2}$ on $M$ are isoresonant if their
resonance sets are the same including the multiplicities.
\end{definition}

\begin{remark} The scattering phases of two isoresonant surfaces with cusps $(M,g_{1})$ and $(M,g_{2})$
are the same. This follows from Theorem 3.31 in \cite{Mu2}, that expresses the determinant of the scattering matrix as the
Weierstrass product:
\begin{equation}
\phi(s)=\phi(1/2)q^{s-1/2}\prod_{\rho} {\frac{s-1+\bar{\rho}}{s-\rho}},
\label{eq:scattphaseasWp}\end{equation}
where $\rho$ runs over all poles of $\phi(s)$, counted with the
order and $q$ is a well determined constant. Indeed, equation (5.17) in \cite{Mu2}
implies that the constant $q$ is determined by the resonance set.
\end{remark}

\begin{proposition} Let $(M,g_{1})$ and $(M,g_{2})$ be two surfaces with cusps that
are isoresonant. Let $\bar{\Delta}_{a,0}$ be the Laplacian given in Definition \ref{def:Laplmcuspas}
for any $a=(a_{1},\dots,a_{m})$ with $\min\{a_{j}, 1\leq j\leq m\}$ big enough.
Then the corresponding relative heat traces coincide, i.e.,
\begin{equation}\tr(e^{-t\Delta_{g_{1}}}-e^{-t\bar{\Delta}_{a,0}}) =
\tr(e^{-t\Delta_{g_{2}}}-e^{-t\bar{\Delta}_{a,0}}), 
\label{eq:comprhtftcm}\end{equation}
and so do the relative determinants:
\begin{equation}\det(\Delta_{g_{1}},\bar{\Delta}_{a,0})= \det(\Delta_{g_{2}},\bar{\Delta}_{a,0}
).\label{eq:compreldetfisrswc}\end{equation} \label{prop:eqrhtisrm}
\end{proposition}

\begin{proof}
The proof of this proposition follows straight forward from the results of \cite{Mu2}. Let $(M,g_{1})$ and $(M,g_{2})$ be isoresonant.
The trace formula for the relative heat operators in (\cite[eq. (2.2)]{Mu2}) establishes:
\begin{multline}
\tr(e^{-t\Delta_{g}}-e^{-t\bar{\Delta}_{a,0}}) =
\int_{M}(K_{g}(z,z,t)-\sum_{j=1}^{m} p_{a_{j}}(z,z,t))
dA_{g}(z)\\ = \sum_{k} e^{-\lambda_{k}t} -
{\frac{1}{4\pi}}\int_{-\infty}^{\infty} e^{-(1/4+\lambda^{2})t}
{\frac{\phi'}{\phi}}(1/2 + i\lambda) d\lambda\\ + {\frac{1}{4}}
e^{-t/4} (\tr(C_{g}(1/2))+m) + {\frac{e^{-t/4}}{\sqrt{4\pi t}}}
\sum_{j=1}^{m} \log(a_{j}), \label{eq:relheatteq22Mu2}
\end{multline}
where $p_{a_{j}}(z,z,t))$ is given by equation (\ref{eq:psuba}).
The term ${\frac{m}{4}} e^{-t/4}$ on the right hand side of (\ref{eq:relheatteq22Mu2})
is missing in (\cite[eq. (2.2)]{Mu2}) because of a missprint. This term comes from the boundary condition of the
model operator $\bar{\Delta}_{a,0}$. Now, by Theorem $5.11$ in \cite{Mu2} the integral
that involves the logarithmic derivative of the scattering matrix
can be rewritten as follows:
\begin{multline}
- {\frac{1}{4\pi}}\int_{-\infty}^{\infty} e^{-(1/4+\lambda^{2})t}
{\frac{\phi'}{\phi}}(1/2 + i\lambda) d\lambda =
{\frac{-\log(q)}{(4\pi)^{3/2}}} {\frac{e^{-t/4}}{\sqrt{t}}}\\  +
{\frac{1}{4}} \sum_{\rho} n(\rho)\{ e^{-t{\rho}(1-{\rho})}
\text{Erfc}((\sqrt{t}(1/2-\rho)) +
e^{-t\bar{\rho}(1-\bar{\rho})} \text{Erfc}(\sqrt{t}(1-\bar{\rho})))\},
\label{eq:Mustracef}\end{multline} where $\rho$ runs over all zeros
and poles of $\phi(s)$ in $\Re(s)<1/2$, $n(\rho)$ denotes either the
order of the pole $\rho$ or the negative of the order of the zero
$\rho$, $q$ is the same constant as in equation (\ref{eq:scattphaseasWp}), and Erfc is the
complementary error function, see \cite[(5.13)]{Mu2}.
In addition, it is clear that the eigenvalues of the Laplacians coincide.
Then, equations
(\ref{eq:relheatteq22Mu2}) and (\ref{eq:Mustracef}) imply equation (\ref{eq:comprhtftcm}).
Equation (\ref{eq:compreldetfisrswc}) follows straightforward from the definition of the relative determinant.
\end{proof}

We are ready to state the main theorem of this section:
\begin{theorem}
Let $(M,g)$ be a surface with cusps and with $\chi(M)<0$, let $K\subset M$ be a fixed
compact subset of $M$ and let $[g]_{K}$ be the
$K$-compactly supported conformal class of $g$. Then isoresonant
sets in $[g]_{K}$ are compact in the $C^{\infty}$-topology.
\label{theorem:compactnessisospectral}
\end{theorem}

The proof of the theorem consists in reducing to the compact case and
apply the result of OPS in \cite{OPS2} restricted to a conformal class.

\begin{proof}
First of all we need to compactify $M$ to a Riemannian manifold that
contains $K$ iso\-me\-tri\-ca\-lly. It is convenient at this point to change
coordinates in the cusp, we first identify $z=(y,x)\in [\alpha,\infty)\times S^{1}$ with
$z = x + iy \in S^{1}\times i [\alpha,\infty) \subset \C$
and then we apply the transformation $z\to w=e^{iz}$.
Then $Z_{\alpha}$ becomes $\{w\in \C : 0 < \vert w\vert \leq
e^{-\alpha}\}=:D^{*}_{e^{-\alpha}}$ and the metric on it becomes
$$\left. g\right\vert_{D^{*}_{e^{-\alpha}}}=\log(\vert
w\vert^{-1})^{-2} \vert w\vert^{-2} \vert dw\vert^{2}.$$ Let us keep
the old notation in these new coordinates. Then for any $b\geq \alpha$,
$M_{b}=M_{0}\cup (D^{*}_{e^{-\alpha}} \setminus D^{*}_{e^{-b}})\cup
\Sigma_{b}$ and we could also denote $D^{*}_{e^{-b}}$ by $Z_{b}$.
Let $\beta > \alpha$ be fixed, as it was explained after equation
(\ref{eq:dKsccm}). Let $f\in C^{\infty}(M)$ satisfy:
\begin{equation}f(w):=
\begin{cases}
\vert \log(\vert w\vert)\vert \vert w\vert & \mbox{ if } w\in
D^{*}_{e^{-\beta-2}}(\cong Z_{\beta+2})\cr
 1 & \mbox{ if } w\in M_{\beta+1}, \cr
\end{cases}\label{eq:funcfcpp} \end{equation}
and put:
\begin{equation}\sigma = f(z)^{2}\cdot g.\label{eq:auxmetsigma}\end{equation}
\noindent Then take
$\widetilde{M} = M\cup \{0\}$ the one-point compactification of $M$
($m$-point compactification if $M$ has $m$ cusps). The metric
$\sigma$ on $M$ extends to a smooth metric on $\wt{M}$ which we
denote again by $\sigma$. Thus $(\wt M,\sigma)$ is a closed manifold
that
contains $M_{\beta}$
isometrically and that has the same genus as $M$. In particular, $K\subset M_{0}\cup (D^{*}_{e^{-\alpha}} \setminus
D^{*}_{e^{-\beta}})$.

Now let $\{g_{n}\}_{n\in \N} \subset [g]_{K}$ be a sequence of
isoresonant metrics. Notice that since
the metrics in the sequence are isoresonant, they have all the same zeroth
 heat invariant, therefore their areas $A_{g_{n}}$
have the same value. Since $g_{n}\in [g]_{K}$, there exists a
function $\varphi_{n}\in C_{c}^{\infty}(M)$ such that $g_{n}=
e^{2\varphi_{n}}g$ and $\supp \varphi_{n} \subset K$, for each $n\in
\N$. Now put: $$\widetilde{g}_{n} := e^{2\varphi_{n}}\sigma.$$ Then
the metrics $\wt{g}_{n}$ are conformal to $\sigma$ on
$\widetilde{M}$. The fact that $K \varsubsetneq
M_{\beta+1}=M\setminus D^{*}_{e^{-\beta-1}}$ and $\left. \sigma
\right\vert_{M_{\beta +1}} = \left. g \right\vert_{M_{\beta +1}}$
imply that the values $A_{g_n} - A_{g}(D^{*}_{e^{\beta+1}})$ are
constant. Then the areas $A_{\wt{g}_n}$ of $(\wt{M},\wt{g}_{n})$
have all the same value; this follows from:
$$A_{\wt{g}_n} = A_{g_n} - A_{g}(D^{*}_{e^{\beta+1}}) +
A_{\sigma}(\wt{M}\setminus M_{\beta+1}).$$ Therefore we can
renormalize the metrics $\wt{g}_{n}$ such that $A_{\wt g_{n}}= 1$.

In addition, the definitions of $K$ and $\sigma$,
the condition $\supp \varphi_{n}\subset K$ for all $n\in \N$, and the
locality of the Laplacians $\Delta_{g}$ and $\Delta_{\sigma}$ imply
that
\begin{equation}
\Vert \varphi_{n}\Vert^{2}_{H^{k}(\widetilde{M}, \sigma)}=\Vert
\varphi_{n}\Vert^{2}_{H^{k}(M,g)}\label{eq:eqSobnorncMacM}.\end{equation}

Notice that compactness of $\{\varphi_{n}\}_{n\in \N}$ in
$C^{\infty}(\wt{M},\sigma)$ together with $\supp \varphi_{n}\subset
K\Subset M$ for all $n\in \N$, imply compactness of
$\{\varphi_{n}\}_{n\in \N}$ in $C^{\infty}(M,g)$. Therefore, in order to prove
compactness in $C^{\infty}(M,g)$ we need to prove uniform boundedness of the sequence
$\{\Vert \varphi_{n}\Vert_{H^{k}(\widetilde{M}, \sigma)}\}_{n\in \N}$ of
the $k$-th Sobolev norms for each $k \geq 1$, i.e. we need to prove that
$$\Vert \varphi_{n}\Vert_{H^{k}(\widetilde{M}, \sigma)} \leq C(k) \ \ \ \ \mbox{
for all } n\in \N,\label{eq:auxubSnsf}$$
where $C(k)$ is a constant that may depend on $k$.

In Lemmas \ref{lemma:consdetascompmet} and
\ref{lemma:eqhinvcorcompm} we prove that if $\{g_{n}\}_{n\in \N}$ is
isoresonant then $\det \Delta_{\widetilde{g}_{n}}$ is constant and
the heat invariants of the metrics $\wt{g}_{n}$ are the same for all
$n$.

Then the theorem follows from the results of OPS in \cite{OPS2} since
the uniform bound of the Sobolev norms of the
functions $\{\varphi_{n}\}_{n\in \N}$ in $(\wt{M},\sigma)$, res\-tric\-ted to our case,
only requires that the determinants, $\det \Delta_{\wt{g}_{n}}$,
the areas, $A_{\wt{g}_{n}}$, and the heat invariants,
$a_{j}(g_{n})$, are constants independent of $n$.
\end{proof}

\begin{lemma} Let $\{g_{n}\}_{n\in \N}$ be a sequence of isoresonant metrics in a
conformal class $[g]_{K}$. Let $\{\wt{g}_{n}\}_{n\in \N}$ be the asso\-cia\-ted sequence of
metrics on $\wt{M}$ defined above. Then the regularized determinants $\det \Delta_{\widetilde{g}_{n}}$
are constant, i.e. their value is independent of $n$. \label{lemma:consdetascompmet}
\end{lemma}

\begin{proof}
Let $h$ be any metric in $[g]_{K}$. Remember the construction we did in the proof of
Theorem \ref{theorem:compactnessisospectral}.
Recall that $\wt{M} = M\cup \{0\}$, the
one-point compactification of $M$, is endowed with a smooth Riemannian metric $\sigma$ obtained from equation
(\ref{eq:funcfcpp}). Let $\wt{h}$ be the metric on $\wt{M}$ corresponding to $h$ via the process described in the proof of
Theorem \ref{theorem:compactnessisospectral}.
Then for the relative determinant
of $(\Delta_{h},\Delta_{\beta,0})$ and the determinant of
$\Delta_{\wt{h}}$ we have the following splitting formulas:
\begin{equation*}
\log \det(\Delta_{h},\Delta_{\beta,0}) - \log \det\Delta_{Z_{\beta},1} - \log
\det \Delta_{(M_{\beta},h),D}
= \log \left({\frac{A_{h}(M)}{\ell(\Sigma_{\beta},h)}}\right) + \log
{\det}^{*}R_{h}
\end{equation*}
and
\begin{multline*}
\log \det \Delta_{(\widetilde{M},\widetilde{h})} -
\log \det \Delta_{(M_{\beta},\widetilde{h}),D}
- \log \det\Delta_{(\wt{M}\setminus M_{\beta},\widetilde{h}),D}\\
= \log \left({\frac{A_{\wt
h}(\wt{M})}{\ell(\Sigma_{\beta},\widetilde{h})}}\right)
+ \log {\det}^{*} R_{\widetilde{h}},
\end{multline*}
where the first formula was proved in Theorem
\ref{theorem:splitformdets} (equation (\ref{eq:splitformallop})), and the second formula is the well
known splitting formula for a closed surface, as in Burghelea, Friedlander and Kappeler \cite{BFK}.
Subtracting the equations we obtain:
\begin{multline*}
\log \det \Delta_{(\widetilde{M},\widetilde{h})} - \log
\det(\Delta_{h},\Delta_{\beta,0})  + \log \det\Delta_{Z_{\beta},1}
- \log \det\Delta_{(\wt{M}\setminus M_{\beta},\widetilde{h}),D}\\
= \log \left({\frac{A_{\wt
h}(\wt{M})}{\ell(\Sigma_{\beta},\widetilde{h})}}\right) -\log
\left({\frac{A_{h}(M)}{\ell(\Sigma_{\beta},h)}}\right) + \log
{\det}^{*} R_{\widetilde{h}} -\log {\det}^{*}R_{h}.
\end{multline*}

From the definition of $f$ we have that $\widetilde{h}=h$ on
$M_{\beta+1}$, and $f\equiv 1$ in a neighborhood of
$\Sigma_{\beta}$. So we have that
$\ell(\Sigma_{\beta},h)=\ell(\Sigma_{\beta},\widetilde{h})$.

Now, let $\{g_{n}\}_{n\in \N}$ be a sequence of isoresonant metrics in
$[g]_{K}$ satisfying the hypothesis of this lemma, and let $\{\wt{g}_{n}\}_{n\in \N}$ be the corresponding sequence in $\wt{M}$.
If we take $h=g_{n}$, the Dirichlet-to-Neumann operators are the same for all $n$.
To see this, notice that given a function $u\in
C^{\infty}(\Sigma_{\beta})$, the unique solution to the problem
$\Delta_{g} \wt{u}=0$ on $M\setminus \Sigma_{\beta}$ with
$\wt{u}\vert_{\Sigma_{\beta}}=u$ will also be a solution of
$\Delta_{g_{n}} \wt{u}=e^{-2\varphi_{n}}\Delta_{g} \wt{u}=0$ on $M\setminus
\Sigma_{\beta}$ satisfying the same boundary condition. Then, it follows from Lemma \ref{lemma:DtNm0cusp}, Proposition
\ref{prop:limitDtNo}, and the fact that the metrics coincide in a neighborhood of the curve $\Sigma_{\beta}$
that the operators $R_{g_{n}}$ are the same for all $n$. Therefore,
${\det}^{*}R_{g_{n}}=c_{1}$, for all $n\in \N$. The same argument applied to the sequence $\{\wt{g}_{n}\}_{n\in \N}$
gives ${\det}^{*} R_{\wt{g}_{n}} = c_{2}$, for all $n\in \N$. In this way, we obtain:
\begin{multline*}
\log \det \Delta_{(\widetilde{M},\widetilde{g}_{n})} - \log
\det(\Delta_{g_{n}},\Delta_{\beta,0}) - \log
\det\Delta_{(\wt{M}\setminus M_{\beta},\widetilde{g}_{n}),D}\\ = \log
(A_{\wt{g}_{n} }(\wt{M})) -\log (A_{g_{n}}(M)) + c
\end{multline*}
where $c$ is a constant that does not depend on $n$.

Recall that
$A_{g_{n}}(M)$, $A_{\wt{g}_{n}}(\wt{M})$,
$\det(\Delta_{g_{n}},\Delta_{\beta,0})$ are constants independent of
$n$. Moreover, $\wt{g}_{n}\vert_{\wt{M}\setminus
M_{\beta}}=\sigma\vert_{\wt{M}\setminus M_{\beta}}$. Therefore
$\det\Delta_{(\wt{M}\setminus M_{\beta},\widetilde{g}_{n}),D}$ is
also constant. Thus,
\begin{equation*}
\log \det(\Delta_{\widetilde{g}_{n}})  = \text{constant}.
\end{equation*}
\end{proof}

\begin{lemma}
The heat invariants corresponding to the metrics of the sequence
$\{\wt{g}_{n}\}_{n\in \N}$ are the same for any $n\in \N$ if we
start with an isoresonant sequence $\{g_{n}\}_{n\in \N}$.
\label{lemma:eqhinvcorcompm}
\end{lemma}

\begin{proof}
Let $h$ be any of the metrics $g_{n}$ that we are considering. Let us
start by constructing the kernel of a parametrix $H_{h}$ for the
heat operator $e^{-t\Delta_{h}}$ on the surface with cusps $(M,h)$, as it was done in
\cite[p.245]{Mu}. Namely we use
the standard method of gluing the heat kernel on the complete
hyperbolic cusp $(0,\infty)\times S^{1}$, denoted by $K_{1}$ and
independent of the choice of $h$, with the heat kernel on $(\wt{M}, \wt{h})$,
denoted by $K_{2,\wt{h}}$, restricted to $M_{\beta +2}$.
Let us recall briefly the definition of the gluing functions: For any two
constants $1<b<c$, let $\phi_{(b,c)}$ be such that $\phi_{(b,c)}(y,x) = 0$ for
$y\leq b$, and $\phi_{(b,c)}(y,x) = 1$ for $y\geq c$. Let $\psi_{1}
= \phi_{(\beta+{\frac{5}{4}}, \beta+2)}$, and $\psi_{2} = 1 -
\psi_{1}$; then $\{\psi_{1},\psi_{2}\}$ is a partition of unity on
$[\beta+1, \beta+2]\times S^{1}$. Let $\phi_{1} =
\phi_{(\beta,\beta+1)}$ and $\phi_{2} = 1 -
\phi_{(\beta+{\frac{5}{2}}, \beta+3)}$, so that $\phi_{i}=1$ on the
support of $\psi_{i}$, $i=1,2$. Then the function:
$$H_{h}(z,z',t) = \phi_{1}(z)K_{1}(z,z',t)\psi_{1}(z') +
\phi_{2}(z)K_{2,\wt{h}}(z,z',t)\psi_{2}(z').$$
is a parametrix, see \cite{Mu}. It is not difficult to prove that there exist constants
$C,c>0$ such that:
$$\int_{M} \vert K_{h}(z,z,t)-H_{h}(z,z,t) \vert dA_{h}(z) \leq C
e^{-{\frac{c}{t}}}$$
for $0<t\leq 1$, see for example \cite{Aldana}. Then for small $t$ we can replace the heat kernel
$K_{h}$ for the parametrix $H_{h}$. Let $p_{\beta+1}(z,z',t)$ be as in equation (\ref{eq:psuba}).
We can obtain the analog to equation ($8.14$) in \cite[p.283]{Mu},
exactly in the same way as it is done there, this is:
\begin{multline}\int_{M} (K_{h}(z,z,t) - p_{\beta+1}(z,z,t)) \ dA_{h}(z)
= \int_{Z_{\beta+1}} (K_{1}(z,z,t) - p_{\beta+1}(z,z,t)) \ dA_{h}(z)\\
+ \int_{M_{\beta+1}} K_{2,\wt{h}}(z,z,t) \ dA_{h}(z)
+O(e^{-{\frac{c}{t}}}), \text{ as } t\to 0. \label{eq:auxreltraces}
\end{multline}
For the convenience of the reader, we give some explicit steps of the proof of equation (\ref{eq:auxreltraces}).
In order to keep the notation simple, if there is no place to confusion, we drop the variable inside the integrals.
From the definition of the cutoff functions we have that
\begin{multline*}\int_{M} (K_{h} - p_{\beta+1}) \ dA_{h}
= \int_{Z_{\beta+1}} (K_{1}\psi_{1} - p_{\beta+1}) \ dA_{h}\\
+ \int_{M_{\beta+2}} K_{2,\wt{h}}\psi_{2} \ dA_{h} + O(e^{-{\frac{c}{t}}}), \text{ as } t\to 0.
\end{multline*}
On the other hand,
$$\int_{Z_{\beta+1}} (K_{1}\psi_{1} - p_{\beta+1}) \ dA_{h} =
\int_{Z_{\beta+1}} (K_{1} - p_{\beta+1}) \ dA_{h}
- \int_{[\beta+1,\beta+2]\times S^{1}} K_{1}\psi_{2}\ dA_{h}$$
Therefore,
\begin{multline*}\int_{M} (K_{h} - p_{\beta+1}) \ dA_{h}
= \int_{Z_{\beta+1}} (K_{1} - p_{\beta+1}) \ dA_{h} + \int_{M_{\beta+1}} K_{2,\wt{h}} \ dA_{h}\\
- \int_{[\beta+1,\beta+2]\times S^{1}} (K_{1}-K_{2,\wt{h}})\psi_{2}\ dA_{h} + O(e^{-{\frac{c}{t}}}), \text{ as } t\to 0.
\end{multline*}
Now, Proposition 3.24 in \cite{Mu} implies that the
coefficients of the asymptotic expansions of $K_{1}(z,z,t)$ and $K_{2}(z,z,t)$, as $t\to 0$, coincide on
$[\beta+1,\beta+2]\times S^{1}$. Then equation (\ref{eq:auxreltraces}) follows.

For a metric $\wt{g}_{n}$ on $\wt{M}$ the heat invariants are, by
definition, the coefficients in the asymptotic expansion of the
trace of the heat kernel as $t\to 0$:
$$\int_{\wt{M}} K_{2,\wt{g}_{n}}(z,z,t) \ dA_{\wt{g}_{n}}(z) \sim {\frac{1}{t}}
\sum_{j=0}^{\infty} a_{j}(\wt{g}_{n}) t^{j}, \ \ \ \ \mbox{ as }
t\to 0.$$ The goal of this lemma is to prove that $a_{j}(\wt{g}_{n})
= a_{j}(\wt{g}_{m})$ for any $n, m \in \N$, and for all $j\geq 0$.
This will follow from the equality of the asymptotic expansions for
small values of $t$ of the integrals
\begin{equation}\int_{\wt{M}} K_{2,\wt{g}_{n}}(z,z,t) \ dA_{\wt{g}_{n}}(z) \quad
\text{and} \quad
\int_{\wt{M}} K_{2,\wt{g}_{m}}(z,z,t) \
dA_{\wt{g}_{m}}(z)\label{eq:eqtrhkcomppart}
\end{equation} for any $n, m \in \N$. We can split the integral over
$\wt M$ as an integral over $M_{\beta+1}$ and one over $\wt{M}\setminus
M_{\beta+1}$.
Given two metrics $g_{n}$ and $g_{m}$ as in the statement of the
lemma, we have that on $\wt{M}\setminus M_{\beta+1}$,
$\wt{g}_{n}=\wt{g}_{m}$. Since relative to any coordinate system,
the coefficients of the asymptotic expansion of the heat kernel are
given by universal polynomials in terms of the metric tensor and its
covariant derivatives, we have that $a_{j}(z,\wt{g}_{n}) =
a_{j}(z,\wt{g}_{m})$, for $z\in \wt{M}\setminus M_{\beta+1}$. On
$\wt{M}\setminus M_{\beta+1}$ we have that
$dA_{\wt{g}_{n}}=dA_{\wt{g}_{m}}$. Therefore:
$$\int_{\wt{M}\setminus M_{\beta+1}} K_{2,\wt{g}_{n}}(z,z,t) \
dA_{\wt{g}_{n}}(z) = \int_{\wt{M}\setminus M_{\beta+1}}
K_{2,\wt{g}_{m}}(z,z,t) \ dA_{\wt{g}_{n}}(z).$$

By assumption, $K_{1}$ and $p_{\beta+1}$ are independent of
$g_{n}$ and $g_{m}$. Therefore, by equation (\ref{eq:auxreltraces}) we have:
\begin{multline*}
\int_{M_{\beta+1}} K_{2,\wt{g}_{n}} \ dA_{\wt{g}_{n}} -
\int_{M_{\beta+1}} K_{2,\wt{g}_{m}} \ dA_{\wt{g}_{m}} \\
\sim_{t\to 0}
 \int_{M} (K_{g_{n}} - p_{\beta+1}) \ dA_{g_{n}} -
\int_{Z_{\beta+1}} (K_{1} - p_{\beta+1}) \ dA_{g_{n}}\\
  -\int_{M} (K_{g_{m}} - p_{\beta+1}) \ dA_{g_{m}} +
\int_{Z_{\beta+1}} (K_{1} - p_{\beta+1}) \ dA_{g_{m}}\\
= \int_{M} (K_{g_{n}} - p_{\beta+1}) \ dA_{g_{n}} -
\int_{M} (K_{g_{m}} - p_{\beta+1}) \ dA_{g_{m}} = 0,
\end{multline*}
where the last equality follows from the fact that the metrics are
isoresonant and from Proposition \ref{prop:eqrhtisrm}. So, we have proved that the asymptotic
expansions as $t\to 0$ for the integrals in
(\ref{eq:eqtrhkcomppart}) are the same. From the definition of the
heat invariants it follows that:
$$a_{j}(\wt{g}_{n}) = a_{j}(\wt{g}_{m}), \ \ \ \ \ \mbox{ for all } j\geq 0, \
\mbox{ and } n, m\in \N.$$
\end{proof}

\section{Boundedness of the relative determinant as function on the Moduli space of hyperbolic surfaces with cusps}
\label{section:boundedness}

In this section we restrict to surfaces with cusps that are hyperbolic.
Let $(M,\tau)$ be a Riemann surface of genus $q$ with $m$ cusps,
where $\tau$ is a hyperbolic metric of constant negative unitary
curvature. To each element $[\tau] \in \M_{q,m}$ we associate the
relative determinant $\det(\Delta_{\tau},\bar{\Delta}_{1,0})$, where
the operator $\bar{\Delta}_{1,0}$ is given in Definition \ref{def:Laplmcuspas} with $a=1$
and it acts on a subspace of $\oplus_{j=1}^m L^2([1,\infty),y_j^{-2}dy_j)$.
If $(M,\tau)$ can be decomposed as $M = M_{0}\cup Z_{a_{1}} \cup
\cdots Z_{a_{m}}$, with $a_{j}\geq 1$; then the difference
$e^{-t\Delta_{\tau}}-e^{-\bar{\Delta}_{1,0}}$ is taken in the
extended $L^{2}$ space given by:
\begin{multline*}L^{2}(M,dA_{\tau})\oplus \oplus_{j=1}^{m}
L^{2}([1,a_{j}],y^{-2}dy)\\ = L^{2}(M_{0},dA_{\tau})\oplus
\oplus_{j=1}^{m} (L^{2}_{0}(Z_{a_{j}}) \oplus
L^{2}([1,\infty),y^{-2}dy)).\end{multline*}

Our result is:
\begin{theorem}
As function on the moduli space of hyperbolic surfaces of fixed genus $q$ with $m$ cusps, $\M_{q,m}$,
the relative determinant $\det(\Delta_{\tau},\bar{\Delta}_{1,0})$ tends to zero as $[\tau]$
approaches the boundary; where by boundary we mean the set
$\overline{{\M}_{q,m}}\setminus \M_{q,m}$.\label{theorem:vdbMs}
\end{theorem}

This theorem implies that the relative determinant is bounded as a
function on the moduli space. In addition, it also implies that it is a proper function.

We use Selberg's trace formula and the work
of Bers in \cite{Bers} and of Jorgenson and Lundelius in
\cite{JorLund1}. In \cite{JorLund1} the authors define a hyperbolic determinant for
Laplacians on hyperbolic Riemann surfaces of finite volume, non-connected in general.
We compare both determinants and use their results together with the results in \cite{Bers} about degeneration of surfaces.

Let us start by recalling Selberg's trace formula \cite{Selberg} as
it is presented by H. Iwaniec in \cite{Iwaniec}, applied to the function
$h(r)=e^{-t({\frac{1}{4}}+r^{2})}$ and its Fourier transform
$g(u)={\frac{1}{\sqrt{4\pi t}}}e^{-{\frac{t}{4}}}
e^{-{\frac{u^{2}}{4t}}}$.

Let $\Gamma$ be a Fuchsian group of the first kind. Let $\Gamma
\setminus \H =M$ be the associated surface, let $\Delta_{\tau}$ be the
Laplacian on $M$ and let $\lambda_{j} = {\frac{1}{4}} + r_{j}^{2}$
be the sequence of eigenvalues of $\Delta_{\tau}$. We do not
include the contribution of the elliptic elements, because we
consider groups without elliptic elements. In this case Selberg's
trace formula (see \cite{Selberg} or \cite{Iwaniec}) applied to the heat operator takes the form:
\begin{multline}
\sum_{j} e^{-t({\frac{1}{4}}+r_{j}^{2})} - {\frac{1}{4\pi}}
\int_{\R} e^{-t({\frac{1}{4}}+\lambda^{2})}
{\frac{\phi'}{\phi}}({\frac{1}{2}}+ i\lambda) d\lambda +
{\frac{e^{-{\frac{t}{4}}}}{4}} \tr(C({\frac{1}{2}})) \\ =
{\frac{\mbox{Area}(M)}{4\pi}} \int_{\R}
e^{-t({\frac{1}{4}}+\lambda^{2})} \lambda \tanh(\pi \lambda)
d\lambda + {\frac{e^{-{\frac{t}{4}}}}{\sqrt{4\pi t}}}
\sum_{k=1}^{\infty} \sum_{{\{\gamma\}}_{\Gamma}}
{\frac{\ell(\gamma)}{2 \sinh{({\frac{k \ell(\gamma)}{2}})}}}
e^{-{\frac{(k\ell(\gamma))^{2}}{4t}}} \\   - {\frac{m}{\pi}}
\int_{\R} e^{-t({\frac{1}{4}}+\lambda^{2})}
{\frac{\Gamma'}{\Gamma}}(1+i\lambda) d\lambda + {\frac{m}{4}}
e^{-{\frac{t}{4}}} - m \log(2) {\frac{e^{-{\frac{t}{4}}}}{\sqrt{4\pi
t}}}, \quad \label{eq:Selbergstf}
\end{multline}
where the sum runs over the primitive hyperbolic conjugacy classes
$\gamma$ with length $\ell(\gamma)$, $m$ is the number of
inequivalent cusps, and as before $C(s)$ is the scattering matrix and $\phi(s)
= \det C(s)$.

In the notation of \cite{JorLund1} the hyperbolic heat trace $\text{HTr}K_{M}(t)$
and the
regularized trace $\text{STr}K_{M}(t)$ are given by:
\begin{align*}
\text{HTr K}_{M}(t) &= {\frac{e^{-{\frac{t}{4}}}}{\sqrt{16\pi t}}}
\sum_{k=1}^{\infty} \sum_{\{\gamma\} \Gamma}
{\frac{\ell(\gamma)}{\sinh{({\frac{k \ell(\gamma)}{2}})}}}
e^{-{\frac{(k\ell(\gamma))^{2}}{4t}}},\\
\text{STrK}_{M}(t) &= \text{HTr K}_{M}(t) + \mbox{Area}(M)
K_{\mathbb H}(t,0),\end{align*} where
$$K_{\mathbb H}(t,0)={\frac{1}{4\pi}} \int_{\R}
e^{-t({\frac{1}{4}}+\lambda^{2})} \lambda \tanh(\pi \lambda)
d\lambda.$$
With these expressions, they
define the hyperbolic zeta function and the hyperbolic determinant as
\begin{eqnarray*}
\zeta_{M, \text{hyp}}(s) &=& {\frac{1}{\Gamma(s)}} \int_{0}^{\infty}
(\text{STrK}_{M}(t)-d) t^{s-1} dt \quad \text{ and }\\
{\det}_{\text{hyp}} \Delta_{\tau} &:=& \exp(-\zeta'_{\text{hyp}}(0)),
\end{eqnarray*}
where $d$ is the number of connected components of $M$ as well as
the dimension of $\ker(\Delta_{\tau})$. Let $Z(s)$ be the Selberg zeta function associated to $M= \Gamma \setminus \H$, then there is
the following relation between the hyperbolic determinant and the derivative at $s=1$ of the Selberg zeta function:
$${\det}_{\text{hyp}} \Delta_{\tau} = Z'_{M}(1) e^{\chi(M) (-2\zeta_{R}'(-1)+ \frac{1}{4} - \frac{\log(2\pi)}{2})},$$
where $\zeta_{R}$ denotes the Riemann zeta function. This formula was proven for the hyperbolic determinant on Riemann
surfaces of finite volume by JL in
\cite{JorLund1}, as a generalization of the corresponding formula on compact Riemann surfaces given in
\cite{DhokerPhong} and \cite{Sarnak}.

We want to see the relation between the hyperbolic determinant
${\det}_{\text{hyp}} \Delta_{\tau}$ and the relative
determinant $(\Delta_{\tau},\bar{\Delta}_{1,0})$. In order to do
that we consider $P(t)$, the contribution of the
parabolic elements to the trace formula. We know that $P(t)$ is
given by
$$P(t)= \int_{\R} e^{-t({\frac{1}{4}}+r^{2})} {\frac{\Gamma'}{\Gamma}}(1+ir)\ dr,$$
for which we have the following lemma:
\begin{lemma}
$P(t)$ has the following asymptotic expansions:
$$P(t)\sim -{\frac{\pi}{2}}{\frac{\log(t)}{t}} + {\frac{\sqrt{\pi}}{2\sqrt{t}}}
(-B_{1} + \gamma -\log(4) + \pi) + t^{-1/2}\sum_{j=1}^{\infty} b_{j}
t^{j/2}, \quad \text{ as } t\to 0,$$ where $B_{1}$ is the first
Bernoulli number and $\gamma$ in this case denotes the Euler
constant. As $t\to \infty$, we have that
$P(t)=O(e^{-{\frac{t}{4}}})$. \label{lemma:asymexppart}
\end{lemma}
\begin{proof}
The proof of Lemma \ref{lemma:asymexppart} easily follows from the
formula
$${\frac{\Gamma'(z+1)}{\Gamma(z+1)}} = {\frac{1}{2z}} +
\log(z) - \int_{0}^{\infty} \left({\frac{1}{2}} -{\frac{1}{u}} +
{\frac{1}{e^{u}-1}}\right) du,$$ for $\Re(z)>0$, and from Stirling's
formula:
$$\log(\Gamma(z)) = (z-{\frac{1}{2}})\log(z) - z + {\frac{1}{2}} \log(2\pi)
+ \sum_{r=1}^{\infty} {\frac{(-1)^{r-1}B_{r}}{2r(2r-1)z^{2r-1}}},$$
for $\vert \arg(z) \vert \leq {\frac{\pi}{2}} - \theta$, where
$B_{r}$ is the $r$-th Bernoulli number.
\end{proof}

\begin{proposition} For the relative determinant and the hyperbolic determinant we have the following relation:
$$\det(\Delta_{\tau},\bar{\Delta}_{1,0}) = \tilde{A}\ {\det}_{\text{hyp}}(\Delta_{\tau}),$$
where $\tilde{A}$ is a constant that depends only on the number of cusps of
$M$. In particular,
$\det(\Delta_{\tau},\bar{\Delta}_{1,0}) = A \ Z'_{M}(1)$, where $A$ depends only on the topology of $M$.
\label{prop:compadets}
\end{proposition}

\begin{proof}
We know that for any $a> 1$ and $t>0$ the operator
$e^{-t\Delta_{a,0}}-e^{-t\Delta_{1,0}}$ acting on
$L^{2}([1,\infty),y^{-2}dy)$ is trace class and the trace is given
by
$$\tr(e^{-t\Delta_{a,0}}-e^{-t\Delta_{1,0}}) = -{\frac{1}{\sqrt{4\pi t}}} \ e^{-t/4} \log(a),$$
see \cite[Prop. 2.6]{Aldana}.
This fact together with equation (\ref{eq:relheatteq22Mu2}), (\cite[eq.(2.2)]{Mu2}), and the linearity of the trace
imply that
$$\tr(e^{-t\Delta_{\tau}} - e^{-t\bar{\Delta}_{1,0}}) = \sum_{j}
e^{-t\lambda_{j}} - {\frac{1}{4\pi}}\int_{\R}
e^{-t({\frac{1}{4}}+r^{2})} {\frac{\phi'}{\phi}}({\frac{1}{2}} + ir)
\ dr + {\frac{e^{-{\frac{t}{4}}}}{4}} (\tr(C({\frac{1}{2}}))+m).
$$

Putting this equation together with Selberg's trace formula we obtain:
\begin{align}
\tr(e^{-t\Delta_{\tau}} - e^{-t\bar{\Delta}_{1,0}}) -
\text{STrK}_{M}(t) = -{\frac{m}{\pi}} P(t) - {\frac{m
\log(2)}{\sqrt{4 \pi t}}}e^{-{\frac{t}{4}}}  + {\frac{m}{2}}
e^{-{\frac{t}{4}}}.\label{eq:difrtStritp}
\end{align}

Let us consider the following auxiliary function:
\begin{equation}\xi(s)={\frac{m}{\Gamma (s)}} \int_{0}^{\infty}
\left\{ -{\frac{1}{\pi}} P(t) + e^{-{\frac{t}{4}}}
\left({\frac{1}{2}} -{\frac{\log(2)}{\sqrt{4 \pi t}}}\right)\right\}
t^{s-1} dt. \label{eq:xiauxfcomrdJLd} \end{equation}
Then we have that $\zeta(s;\Delta_{\tau},\bar{\Delta}_{1,0}) =
\zeta_{M,\text{hyp}}(s) + \xi(s)$. On the other hand,
Lemma \ref{lemma:asymexppart} implies that the function $\xi(s)$ has a
meromorphic continuation to $\C$ that is analytic at $s=0$. Thus,
$$\det(\Delta_{\tau},\bar{\Delta}_{1,0}) = e^{-\xi'(0)} {\det}_{\text{hyp}}(\Delta_{\tau}).$$
The constant $\tilde{A}=e^{-\xi'(0)}$ depends only on the number of cusps of $M$.
\end{proof}

Let us now recall how one can approach the boundary of
the moduli space. For this we refer to L. Bers in \cite{Bers}. Let us recall the notation and the result
in \cite{Bers} that we use here. Let
$G=\SL(2,\R)/\{\pm I\}$. Every Fuchsian group $\Gamma$ satisfying
the condition $\text{mes}(G/\Gamma)<\infty$, has a signature
$\sigma=(p,n;\nu_{1},\cdots,\nu_{n})$, where $p$ and $n$ are
integers, the $\nu_{j}$ are integers or the symbol $\infty$, and
$p\geq 0$, $n\geq 0$, $2\leq \nu_{1}\leq \cdots \leq \nu_{n} \leq
\infty$. In the quotient $\Gamma \backslash {\mathbb H}$, the number $p$ corres\-ponds to the genus and $n$
corresponds to the number of \lq\lq singular\rq\rq \ points. The values $\nu_{j}<\infty$ correspond to elliptic points,
and $\nu_{j} = \infty$ correspond to cusps. Since we do not consider elliptic points, all $\nu_{j}$ are equal to infinity. Let
$$X(\sigma) = \{[\Gamma]: [\Gamma] \text{ is a conjugacy class of Fuchsian groups } \Gamma
\text{ with signature } \sigma\}$$
The spaces $X(\sigma)$, with their natural topologies, are metrizable. The topology of $X(\sigma)$ can be derived
from the Teichm\"uller topology. The theorem that is of our interest is the following:

\begin{theorem}$($L. Bers \cite{Bers}$)$
The subset of $X(\sigma)$ corresponding to groups $\Gamma$ such that
$\ell(\gamma)\geq 2+\epsilon >2$ for all hyperbolic $\gamma \in
\Gamma$ is compact.
\end{theorem}

This implies that the only possible deformations reaching the boundary of the moduli space
are obtained by deforming hyperbolic elements in the group, i.e. by pinching smallest geodesics.

As we already mentioned, the proof of Theorem \ref{theorem:vdbMs} relies strongly on the results of Jorgenson and Lundelius
in \cite{JorLund1}. Let us recall them: Let
$\{M_{l}\}_{l\in I\subset \R_{+}^{p}}$ be a degenerating family of
hyperbolic Riemann surfaces of finite volume (each surface $M_{l}$
is assumed to have $m$ cusps and to be connected) with $p$ pinching geodesics. This
means that for each $l=(l_{1},\cdots, l_{p})\in I$ the cutoff
cylinders $C_{l_{k},\epsilon}$ are embedded in $M_{l}$ for every
$0<\epsilon <1/2$.
From Gauss-Bonnet we know that the area of the surfaces is kept
invariant during the deformation. Let $\Gamma_{l}$ be the group corresponding to $M_{l}$,
let $\text{H}(\Gamma_{l})$ denote a set of representatives of
primitive non-conjugated hyperbolic classes in $\Gamma_{l}$, and let
$\text{DH}(\Gamma_{l})\subset \text{H}(\Gamma_l)$ be the subset corresponding to the geodesics that we are pinching.
Proposition 2.1 in \cite{JorLund1} yields that the degenerating heat
trace for $t>0$ equals:
$$\text{DTr} K_{M_{l}}(t)= {\frac{e^{-t/4}}{\sqrt{16 \pi t}}} \sum_{\text{DH}(\Gamma_{l})}
\sum_{n=1}^{\infty} {\frac{\ell(\gamma)}{\sinh(n\ell(\gamma)/2)}}
e^{-(n\ell(\gamma))^{2}/4t}.$$

Let $M$ be the Riemann surface that is the limit of the degenerating
family $\{M_l\}$ then $M$ is not necessarily connected and the
number of cusps of $M$ is $m+2p$. Theorem 2.2 in \cite{JorLund1}
states that:
$$\lim_{l\to 0}(\text{HTr} K_{M_{l}}(t) - \text{DTr} K_{M_{l}}(t)) = \text{HTr} K_{M}(t).$$

Their next step is to separate (in the trace) the small eigenvalues
of the Laplacian on $M_{l}$.
Let $\{\lambda_{n,l}\}_{n}$ denote the eigenvalues of $\Delta_{M_{l}}$ and $\{\lambda_{j}\}_{j}$
denote the eigenvalues of $\Delta_{M}$.
Let $0<\alpha< 1/4$ be such
that $\alpha$ is not an eigenvalue of the Laplacian on $M$ and consider:
$$\text{HTr} K_{M_{l}}^{\alpha}(t) := \text{HTr} K_{M_{l}}(t) - \sum_{\lambda_{n,l}\leq \alpha} e^{-\lambda_{n,l}t},$$
From this definition we have that:
$\text{STr} K_{M_{l}}^{\alpha}(t) = \text{STr} K_{M_{l}}(t) - \sum_{\lambda_{j,l}\leq \alpha}
e^{-t\lambda_{j,l}},$ and $\text{STr} K_{M}^{\alpha}(t) = \text{HTr} K_{M}(t) -\sum_{\lambda_{j}(M)\leq \alpha}
e^{-t\lambda_{j}(M)} + A K_{\mathbb H}(t,0),$
where $A$ denotes the area of the limit surface $M$.
For a given hyperbolic surface $M_{*}$, JL consider
the truncated hyperbolic zeta function:
$$\zeta_{\text{hyp}\ M_*}^{\alpha}(s) = {\frac{1}{\Gamma(s)}} \int_{0}^{\infty} \text{STr} K_{M_*}^{\alpha}(t) t^{s-1}\ dt$$
and the corresponding determinant ${\det}_{\text{hyp}}^{\alpha} \Delta_{M_*}$ is defined in the usual way.
Let us see
now how ${\det}_{\text{hyp}}^{\alpha} \Delta_{M_l}$ relates to
$\det(\Delta_{M_l},\bar{\Delta}_{1,0})$. Notice that the operator
$\bar{\Delta}_{1,0}$ remains constant through the degeneration.
At the moment we are
not concerned with the relative determinant of the limiting surface
but rather with the behavior
of the relative determinant of the degenerating surfaces. Equation
(\ref{eq:difrtStritp}) applied to $M_{l}$ can be rewritten as:
$$\tr(e^{-t\Delta_{M_l}} - e^{-t\bar{\Delta}_{1,0}}) - \text{STr} K_{M_l}^{\alpha}(t)
= m \left({-\frac{1}{\pi}} P(t) +  ({\frac{1}{2}} -
{\frac{\log(2)}{\sqrt{4\pi t}}}) e^{-{\frac{t}{4}}} \right)
+\sum_{\lambda_{j,l}\leq \alpha} e^{-t\lambda_{j,l}}$$ Writing this
in terms of zeta functions we obtain:
$$\zeta(s,\Delta_{M_l},\bar{\Delta}_{1,0}) - \zeta_{\text{hyp}\
M_l}^{\alpha}(s) = \xi(s) +\sum_{\lambda_{j,l}\leq \alpha}
\lambda_{j,l}^{-s},$$ where $\xi(s)$ is as in equation
(\ref{eq:xiauxfcomrdJLd}). Taking the meromorphic continuations and
differentiating we obtain that:
\begin{equation}
\log {\det}_{\text{hyp}}^{\alpha} \Delta_{M_l} =
\log\det(\Delta_{M_l},\bar{\Delta}_{1,0}) + mc -
\sum_{\lambda_{j,l}\leq \alpha} \log(\lambda_{j,l}),
\label{eq:relrdlhadsev}
\end{equation}
where for $\xi(s)$ we used again Lemma \ref{lemma:asymexppart} and
the fact that from equation (\ref{eq:xiauxfcomrdJLd}) is clear that $\xi'(0)
= c\ m$, where $c$ is a constant independent of $l$.

Due to a missprint in a sign in Corollary $4.3$ in \cite{JorLund1} we do not use it directly. Instead we
refer to their Theorem $4.1$ and keep track
of the signs. Theorem $4.1$ in \cite{JorLund1} establishes that for all $s\in \C$
\begin{equation}\lim_{l\to 0} \big(\zeta_{\text{hyp}\ M_l}^{\alpha}(s) -
{\frac{1}{\Gamma(s)}} \int_{0}^{\infty} \text{DTr}
K_{M_{l}}(t)t^{s-1} dt -\zeta_{\text{hyp}\ M}^{\alpha}(s)\big) = 0,
\label{eq:limlt0zaMlDpZlM}\end{equation}
and the convergence is uniform in any half plane $\Re(s)>C>-\infty$.
In order to deal with the second term in the left-hand side of
equation (\ref{eq:limlt0zaMlDpZlM}) we follow Remark $4.2$ in
\cite{JorLund1} to obtain:
\begin{equation*}\left.{\frac{d}{ds}} {\frac{1}{\Gamma(s)}} \int_{0}^{\infty} \text{DTr} K_{M_{l}}(t)t^{s-1} dt
\right\vert_{s=0} = \sum_{\text{DH}(\Gamma_{l})}
\sum_{n=1}^{\infty}
{\frac{e^{-n\ell(\gamma)}}{n(1-e^{-n\ell(\gamma)})}},
\end{equation*}
This together with equation (\ref{eq:limlt0zaMlDpZlM}) gives:
$$\lim_{l\to 0} \big( \log {\det}_{\text{hyp}}^{\alpha}\Delta_{M_l} +
\sum_{\gamma \in \text{DH}(\Gamma_l)} \sum_{n=1}^{\infty}
{\frac{e^{-n\ell(\gamma)}}{n(1-e^{-n\ell(\gamma)})}} \big) = \log
{\det}_{\text{hyp}}^{\alpha} \Delta_{M}.$$ Let us replace $\log
{\det}_{\text{hyp}}^{\alpha}\Delta_{M_l}$ in the expression above using
equation (\ref{eq:relrdlhadsev}):

\begin{multline}\lim_{l\to 0} \big( \log\det(\Delta_{M_l},\bar{\Delta}_{1,0})
+ mc+ \sum_{\gamma \in \text{DH}(\Gamma_l)} \sum_{n=1}^{\infty}
{\frac{e^{-n\ell(\gamma)}}{n(1-e^{-n\ell(\gamma)})}} -
\sum_{0<\lambda_{j,l} \leq \alpha} \log(\lambda_{j,l})\big)\\
= \log {\det}_{\text{hyp}}^{\alpha}
\Delta_{M}.\label{eq:limlt0logdets1}\end{multline}

In order to study the behavior of $\log
\det(\Delta_{M_l},\bar{\Delta}_{1,0})$ we need to know the behavior
of the series in the left-hand side of equation
(\ref{eq:limlt0logdets1}) as $l\to 0$; recall that $\ell(\gamma)\to
0$ as $l\to 0$. This series was already studied by S.A. Wolpert in \cite[p.308]{Wolpert}, if $\Re(s)>0$ then as $\ell(\gamma)\to 0^{+}$
we have:
$$\sum_{n=1}^{\infty}
{\frac{e^{-ns\ell(\gamma)}}{n(1-e^{-n\ell(\gamma)})}} =
\left({\frac{\pi^2}{6\ell(\gamma)}} + (s-{\frac{1}{2}})\log(1-e^{-s
\ell(\gamma)}) \right) + O(1).$$ Taking
$s=1$ we see that
$$\lim_{l\to 0}\sum_{\gamma \in DH(\Gamma_l)} \sum_{n=1}^{\infty}
{\frac{e^{-n\ell(\gamma)}}{n(1-e^{-n\ell(\gamma)})}} = \infty.$$

For the sum involving the logarithm of the small eigenvalues we know that
some of the small eigenvalues of the family $\{M_{l}\}$
may degenerate. For the eigenvalues of $M$, $0=\lambda_{j}(M)$, that
come from degeneration we know that for any
$0<\alpha<{\frac{1}{4}}$, $\alpha$ not an eigenvalue of $M$, there
is a $l_{0}$ such that for all $0<l\leq l_{0}$, $\lambda_{l,j}\leq
\alpha$. This is due to the convergence of any finite number of
eigenvalues. Thus $\lim_{l\to 0} \sum_{0<\lambda_{j,l} \leq \alpha}
\log(\lambda_{j,l}) = -\infty$. In this way we have:
$$\lim_{l\to 0} \sum_{\gamma \in \text{DH}(\Gamma_l)}
\sum_{n=1}^{\infty}
{\frac{-e^{-n\ell(\gamma)}}{n(1-e^{-n\ell(\gamma)})}}-
\sum_{0<\lambda_{j,l} \leq \alpha} \log(\lambda_{j,l}) = \infty,$$
\noindent since the term $c m$
and the hyperbolic $\alpha$-regularized determinant of the limit surface are both finite, it
follows that
$$\lim_{l\to 0} \log({\det}(\Delta_{M_l},\bar{\Delta}_{1,0})) = -\infty.$$
This finishes the proof of Theorem \ref{theorem:vdbMs}.

\end{document}